\documentclass[14pt,a4paper]{amsart}
\usepackage{amsmath,amsthm,amssymb,latexsym,mathrsfs,amsfonts,cases,indentfirst}
\usepackage{color}

\usepackage{amsmath,amssymb}

\usepackage{amsfonts}
\usepackage{CJK, CJKnumb}
\usepackage{color}
\usepackage{indentfirst}
\usepackage{latexsym, bm}
\usepackage{graphicx}
\usepackage{cases}
\usepackage{fancyhdr}
\usepackage{pifont}
\usepackage{amsmath}

\newtheorem{theorem}{Theorem}[section]
\newtheorem{defi}[theorem]{Definition}
\newtheorem{lemma}[theorem]{Lemma}
\newtheorem{coro}[theorem]{Corollary}
\newtheorem{proposition}[theorem]{Proposition}

\newtheorem{remark}[theorem]{Remark}
\usepackage{enumerate}
\numberwithin{equation}{section}

\begin{document}

\title[$RLL$-realization of   $U_{r,s}\mathcal(\widehat{\mathfrak{so}_{2n}})$]{$RLL$-realization of two-parameter quantum affine algebra in type $D_n^{(1)}$ }

\author[Zhuang]{Rushu Zhuang}
\address{School of Mathematical Sciences, MOE Key Laboratory of Mathematics and Engineering Applications \& Shanghai Key Laboratory of PMMP, East China Normal University, Shanghai 200241, China}
\email{52205500005@ecnu.edu.cn}

\author[Hu]{Naihong Hu{$^*$}}
\address{School of Mathematical Sciences, MOE Key Laboratory of Mathematics and Engineering Applications \& Shanghai Key Laboratory of PMMP, East China Normal University, Shanghai 200241, China}
\email{nhhu@math.ecnu.edu.cn}

\author[Xu]{Xiao Xu}
\address{School of Mathematical Sciences, MOE Key Laboratory of Mathematics and Engineering Applications \& Shanghai Key Laboratory of PMMP, East China Normal University, Shanghai 200241, China}
\email{52275500004@stu.ecnu.edu.cn}

\thanks{$^*$Corresponding author. This work is
	supported by the NNSF of China (Grant No. 12171155), and in part by the Science and Technology Commission of Shanghai Municipality (Grant No. 22DZ2229014).}

\begin{abstract}
We obtain the basic $R$-matrix of the two-parameter Quantum group $U=U_{r,s}\mathcal(\mathfrak{so}_{2n})$ via its weight representation theory and  determine its $R$-matrix with spectral parameters for the two-parameter quantum affine algebra $U=U_{r,s}\mathcal(\widehat{\mathfrak{so}_{2n}})$. Using the Gauss decomposition of the $R$-matrix realization of $U=U_{r,s}\mathcal(\mathfrak{so}_{2n})$, we study the commutation relations of the Gaussian generators and finally arrive at its $RLL$-formalism of the Drinfeld realization of two-parameter quantum affine algebra $U=U_{r,s}\mathcal(\widehat{\mathfrak{so}_{2n}})$.
\end{abstract}

\maketitle

\tableofcontents

\section{Introduction}
Quantum groups were independently discovered by Drinfeld \cite{D1, D2} and Jimbo \cite{J} that the universal enveloping algebra $U(\mathfrak{g})$ of any Kac-Moody algebra $\mathfrak{g}$ admits  a certain $q$-deformation $U_q(\mathfrak{g})$ as a Hopf algebra. Their construction is given
in terms of Chevalley generators and $q$-Serre relations. For the Yangian algebra $Y(\mathfrak{g})$ and the quantum affine algebra $U_q(\widehat{\mathfrak{g}})$
of complex simple Lie algebra $\mathfrak g$, Drinfeld \cite{D3} gave their new realizations, which are quantizations of the loop realizations of the classical loop and affine Lie algebras.  Faddeev, Reshetikhin and Takhtajan \cite{FRT89}  presented the $RLL$-realizations of $U_q(\mathfrak{g})$ (\cite{KS}) of the classical simple Lie algebras $\mathfrak g$ by means of solutions of the quantum Yang-Baxter equation
$$
R_{12} R_{13} R_{23}= R_{23}R_{13}R_{12},
$$
where $R_{12} = R\otimes I$, etc., and $R\in$ End($\mathbb{C}^{n}\otimes\mathbb{C}^{n}$). This realization is a natural analog of the matrix realizations of the classical Lie algebras, which originated from the quantum inverse scattering method developed by the St. Petersburg school. Later on, the $R$-matrix realization of quantum loop algebra $U_q(\mathfrak{g}\otimes\mathbb{C}[t,t^{-1}])$ using a solution of the quantum Yang-Baxter equation with spectral parameters $z, w\in\mathbb{C}$:
$$
R_{12}(z)R_{13}(zw)R_{23}(w)= R_{23}(w)R_{13}(zw)R_{12}(z),
$$
where $R(z)$ is a rational function of $z$ with values in End($\mathbb{C}^{n}\otimes\mathbb{C}^{n}$) was given by Faddeev, Reshetikhin and Takhatajan in \cite{FRT90}.

As we know \cite{Ga}, the affine Kac-Moody algebra $\widehat{\mathfrak{g}}$ admits a natural realization as a central extension of the loop algebra $\mathfrak{g}\otimes\mathbb{C}[t,t^{-1}]$, when $\mathfrak{g}$ is a simple Lie algebra. In \cite{RS}, Reshetikhin and Semenov-Tian-Shansky proved that the central extension can be viewed as the affine analogue of the construction in \cite{FRT90}. In \cite{DF}, by using the Gauss decomposition, Ding and Frenkel  showed that the $R$-matrix realization and Drinfeld realization of quantum affine algebra $U_q(\widehat{\mathfrak{gl}_{n}})$ are isomorphic.
Recently, Jing, Liu and Molev \cite{JLM, JLM1} presented the isomorphisms between the $R$-matrix and Drinfeld presentations
of one-parameter quantum affine algebras $U_q(\widehat {\mathfrak g})$ of affine types $B^{(1)}_n, C^{(1)}_n, D^{(1)}_n$.

On the other hand, in 2001, Benkart and Witherspoon \cite{BW}, motivated by the two-parameters $(r,s)$-Serre relations satisfied by the up-down operators defined on posets, re-obtained the two-parameter quantum enveloping algebras $U_{r,s}(\mathfrak{g})$ corresponding to the general linear Lie algebra $\mathfrak {gl}_n$ and the special linear Lie algebra $\mathfrak {sl}_n$, which were earlier defined by Takeuchi \cite{T} in 1990. Bergeron, Gao, Hu in \cite{BGH} found the defining structures of two-parameter quantum groups $U_{r,s}(\mathfrak g)$ of orthogonal and symplectic Lie algebras, which are realized as the Drinfeld doubles and established their weight representation theory of category $\mathcal O^{r,s}$ (\cite{BGH1}).  Much work has been done for the other types including the affine types (see \cite{Chen-Hu-Wang} and references therein).  Hu-Rosso-Zhang \cite{HRZ} and Hu-Zhang \cite{HZ} initiated to define and study the vertex representations of Drinfeld realizations of two-parameter quantum affine algebras of untwisted types and constructed the quantum affine Lyndon bases. The $RLL$-realization of the two-parameter quantum affine algebra $U_{r,s}(\widehat{\mathfrak{gl}_{n}})$ has been given by Jing and Liu \cite{JL} under the name of $RTT$-realization.

A natural open question is how to work out the $RLL$-realizations for the quantum affine algebras $U_{r,s}(\widehat {\mathfrak g})$ of types $B_n^{(1)}, C^{(1)}_n, D^{(1)}_n$.
The difficulty lies in that there has been no information about two-parameter basic $R$-matrices in the corresponding cases for many years, let alone their Yang-Baxterizations. Recently we have overcome such obstacles and solve these problems (also see the subsequent preprints \cite{HXZ, ZhHJ}).

The current paper is the first to give the $RLL$-realization of the two-parameter quantum affine algebra $U_{r,s}(\widehat{\mathfrak{so}_{2n}})$ by the Reshetikhin and Semenov-Tian-Shanski method. We show that Drinfeld's construction can be naturally established in the Gaussian decomposition of a matrix composed of elements of the quantum affine algebra. The organization of the paper is as follows. In Section 2, we recall the basic results. In Section 3, we give the $\hat{R}$-matrix of the two-parameter quantum group $U_{r,s}(\mathfrak{so}_{2n})$. In Section 4, we give the isomorphism between Faddeev-Reshetikhin-Takhtajan and Drinfeld-Jimbo definitions of $U_{r,s}\mathcal(\mathfrak{so}_{2n})$, and further give the spectral parameter dependent $R$-matrix $\hat{R}(z)$ as the Yang-Baxterzation of the basic $R$-matrix we obtained. In Sections 5--6, we study the commutation relations between Gaussian generators and give the Drinfeld realization of $U_{r,s}(\widehat{\mathfrak{so}_{2n}})$ (modified version of \cite{HZ}).

\section{Preliminaries}
In \cite{BGH}, let $\mathbb{K}=\mathbb{Q}(r,s)$ be a ground field of rational functions in $r, s$, where $r, s$ are algebraically independent indeterminates. Assume $\Phi$ is a finite root system of type $D_{n}$ with $\Pi$ a base of simple roots. Regard $\Phi$ as a subset of a Euclidean space $E=\mathbb{R}^{n}$ with an inner product $(,)$. Let $\epsilon_1,\dots,\epsilon_n$ denote an orthonormal basis of $E$, and suppose $\Pi=\{\alpha_i=\epsilon_i-\epsilon_{i+1}\mid 1\le i<n\}$ $\cup$ $\{\alpha_n=\epsilon_{n-1}+\epsilon_{n}\}$ and $\Phi=\{\pm\epsilon_i \pm \epsilon_j\mid 1\le i\neq j\le n\}$. In this case, set $r_i=r^{(\alpha_i,\alpha_i)/2}$ and $s_i=s^{(\alpha_i,\alpha_i)/2}$, so that $r_1=\cdots=r_n=r$ and $s_1=\cdots=s_n=s$.

The Cartan matrix of $D_n$ is
\begin{equation}
D_n=(a_{ij})_{n\times n}=\left(\begin{array}{rrrrrrr}2&-1&0&\cdots&0&0&0\\-1&2&-1&\cdots&0&0&0\\0&-1&2&\cdots&0&0&0\\\vdots&\vdots&\vdots&\quad&\vdots&\vdots&\vdots\\0&0&0&\cdots&2&-1&-1
\\0&0&0&\cdots&-1&2&0\\0&0&0&\cdots&-1&0&2\end{array}\right).
\end{equation}
The quantum structural constant matrix is of the form:
\begin{equation}
(\langle w^{\prime}_i,w_j\rangle)_{n\times n}=\left(\begin{array}{rrrrrrr}rs^{-1}&r^{-1}&1&\cdots&1&1&1\\s&rs^{-1}&r^{-1}&\cdots&1&1&1\\1&s&rs^{-1}&\cdots&1&1&1\\\vdots&\vdots&\vdots&\quad&\vdots&\vdots&\vdots\\1&1&1&\cdots&rs^{-1}&r^{-1}&r^{-1}
\\1&1&1&\cdots&s&rs^{-1}&r^{-1}s^{-1}\\1&1&1&\cdots&s&rs&rs^{-1}\end{array}\right).
\end{equation}
In \cite{BGH}, let $U_{r,s}\mathcal(\mathfrak{so}_{2n})$ be the unital associative algebra over $\mathbb{Q}(r,s)$ generated by $e_i, f_i,w^{\pm1}_{i},w^{\prime\pm1}_{i}(1\le i\le n)$, subject to the defining relations $(\rm D1)$ --- $(\rm D6)$:
\smallskip

\noindent
(D1) $w^{\pm1}_{i}, w^{\prime\pm1}_{i}$ all commute with one another and $w_i w^{-1}_i=w^{\prime}_{i}w^{\prime-1}_{i}=1$.
\smallskip

\noindent
(D2) For $1\le i\le n,$ $1\le j\le n,$ and $1\le k(\neq n-1)\le n,$ we have
\begin{gather*}
w_{j}e_{i}=r^{(\epsilon_j,\alpha_i)}s^{(\epsilon_{j+1},\alpha_i)}e_{i}w_{j},\qquad w_{j}f_{i}=r^{-(\epsilon_j,\alpha_i)}s^{-(\epsilon_{j+1},\alpha_i)}f_{i}w_{j},\\
w^{\prime}_{j}e_{i}=s^{(\epsilon_j,\alpha_i)}r^{(\epsilon_{j+1},\alpha_i)}e_{i}w^{\prime}_{j},\qquad w^{\prime}_{j}f_{i}=s^{-(\epsilon_j,\alpha_i)}r^{-(\epsilon_{j+1},\alpha_i)}f_{i}w^{\prime}_{j},\\
w_{n}e_{k}=r^{(\epsilon_{n-1},\alpha_k)}s^{-(\epsilon_{n},\alpha_k)}e_{k}w_{n},\quad w_{n}f_{k}=r^{-(\epsilon_{n-1},\alpha_k)}s^{(\epsilon_{n},\alpha_k)}f_{k}w_{n},\\
w^{\prime}_{n}e_{k}=s^{(\epsilon_{n-1},\alpha_k)}r^{-(\epsilon_{n},\alpha_k)}e_{k}w^{\prime}_{n},\quad w^{\prime}_{n}f_{k}=s^{-(\epsilon_{n-1},\alpha_k)}r^{(\epsilon_{n},\alpha_k)}f_{k}w^{\prime}_{n}.
\end{gather*}
(D3)
\begin{gather*}
w_{n}e_{n-1}=r^{(\epsilon_{n},\alpha_{n-1})}s^{-(\epsilon_{n-1},\alpha_{n-1})}e_{n-1}w_{n},\\ w_{n}f_{n-1}=r^{-(\epsilon_{n},\alpha_{n-1})}s^{(\epsilon_{n-1},\alpha_{n-1})}f_{n-1}w_{n},\\
w^{\prime}_{n}e_{n-1}=s^{(\epsilon_{n},\alpha_{n-1})}r^{-(\epsilon_{n-1},\alpha_{n-1})}e_{n-1}w^{\prime}_{n},\\ w^{\prime}_{n}f_{n-1}=s^{-(\epsilon_{n},\alpha_{n-1})}r^{(\epsilon_{n-1},\alpha_{n-1})}f_{n-1}w^{\prime}_{n}.
\end{gather*}
(D4) For $1\le i,\;j\le n$, we have
\begin{gather*}
[e_i,f_j]=\delta_{ij}\frac{w_i-w^{\prime-1}_i}{r-s}.
\end{gather*}
(D5) For any $1\le i\neq j\le n$ but $(i,j)\notin\{(n-1,n),(n,n-1)\}$ with $a_{ij}=0,$ we have:
\begin{gather*}
[e_i,e_j]=0,\quad e_{n-1}e_n=rse_ne_{n-1},\\
[f_i,f_j]=0,\quad f_nf_{n-1}=rsf_{n-1}f_n.
\end{gather*}
(D6) For $1\le i<j\le n$ with $a_{ij}=-1$, we have $(r,s)$-Serre relations:
\begin{gather*}
e^{2}_{i}e_{j}-(r+s)e_{i}e_{j}e_{i}+rse_je^2_i=0,\\
f_{j}f^2_i-(r+s)f_{i}f_{j}f_{i}+rsf^2_if_j=0,\\
e^{2}_{j}e_{i}-(r^{-1}+s^{-1})e_{j}e_{i}e_{j}+r^{-1}s^{-1}e_ie^2_j=0,\\
f_{i}f^2_{j}-(r^{-1}+s^{-1})f_{j}f_{i}f_{j}+r^{-1}s^{-1}f^2_jf_i=0.
\end{gather*}

\begin{proposition}
The algebra $U_{r,s}\mathcal(\mathfrak{so}_{2n})$ is a Hopf algebra with the comultiplication $\Delta$, counit $\varepsilon$ and antipode $S$ such that
\begin{equation*}
\begin{aligned}
\Delta(e_{i})&=e_{i}\otimes 1+w_{i}\otimes e_{i},&\qquad
\Delta(f_{i})&=1\otimes f_i+f_i\otimes w^{\prime}_{i},\\
\varepsilon(e_{i})&=0,&\quad
\varepsilon(f_{i})&=0,\\
S(e_{i})&=-w_i^{-1}e_{i},&\quad
S(f_{i})&=-f_{i}w_i^{\prime-1},
\end{aligned}
\end{equation*}
and $w_i,w^{\prime}_i$ are group-like elements for any $i\in I$.
\end{proposition}
\begin{defi}
\rm We define the linear mapping $f: \Lambda\times \Lambda \rightarrow k^{*}$
$$f(\lambda,\mu)=\langle w^{\prime}_{\mu},w_{\lambda}\rangle^{-1},$$
which satisfies
$$f(\lambda+\nu,\mu)=f(\lambda,\mu)f(\nu,\mu),\quad f(\lambda,\mu+\nu)=f(\lambda,\mu)f(\lambda,\nu).$$
\end{defi}
\begin{defi}
\rm  Let $M, M^{\prime}$ be finite dimensional $U$-modules. Define an isomorphism
 of $U$-modules  $\widetilde{f}: M \otimes M^{\prime} \rightarrow M\otimes M^{\prime}$ as
 $$\widetilde{f}(m\otimes m^{\prime})=f(\lambda,\mu)m\otimes m^{\prime},$$
 where $m\in M_{\lambda}$, $m^{\prime}\in M^{\prime}_{\mu}$ and $\lambda,\mu\in\Lambda$.
\end{defi}
\begin{coro} \cite{BGH1} $U_{r,s}(\mathfrak g)\cong U_{r,s}(\mathfrak
n^-)\otimes U^0\otimes U_{r,s}(\mathfrak n)$, {\it as vector spaces. In
particular, it induces $U_q(\mathfrak g)\cong U_q(\mathfrak
n^-)\otimes U_0\otimes U_q(\mathfrak n)$, as vector
spaces.\hfill}
\end{coro}
Let $Q=\mathbb Z\Phi$ denote the root lattice and set
$Q^+=\sum_{i=1}^n\mathbb Z_{\ge0}\alpha_i$. Then for any
$\zeta=\sum_{i=1}^n\zeta_i\alpha_i\in Q$, we denote
$$
\omega_\zeta=\omega_1^{\zeta_1}\cdots\omega_n^{\zeta_n}, \qquad
\omega_\zeta'=(\omega_1')^{\zeta_1}\cdots(\omega_n')^{\zeta_n}.
$$
Then $U_{r,s}(\mathfrak n^\pm)=\bigoplus_{\eta\in
Q^+}U_{r,s}^{\pm\eta}(\mathfrak n^\pm)$  $Q^\pm$-graded, where
$$
U_{r,s}^{\eta}(\mathfrak n^\pm)=\left\{\,a\in U_{r,s}(\mathfrak
n^\pm)\;\left|\; \omega_\zeta\,a\,\omega_\zeta^{-1}=\langle
\omega_\eta',\omega_\zeta\rangle\,a, \ \omega_\zeta'\,a\,{\omega_\zeta'}^{-1}=\langle
\omega_\zeta',\omega_\eta\rangle^{-1} \,a\,\right\}\right.,
$$
for $\eta\in Q^+\cup Q^-$.
\begin{lemma}
\cite{BGH1} Set $d_{\zeta} = \text{dim}_{\mathbb K} U_{r,s}^{+\zeta}(\mathfrak n^+)$. Consider the basis $\{u^{\zeta}_k\}^{d_\zeta}_{k=1}$ for $U_{r,s}^{+\zeta}(\mathfrak n^+)$, and $\{v^{\zeta}_k\}^{d_\zeta}_{k=1}$ is the dual basis for $U_{r,s}^{-\zeta}(\mathfrak n^-)$ with respect to the pairing. Now let
\begin{equation}
\Theta_{\zeta}=\sum\limits_{k=1}^{d_{\zeta}}v^{\zeta}_k\otimes u^{\zeta}_k\in U\otimes U.
\end{equation}
All but finitely many terms in this sum will act as multiplication by $0$ on any weight space $M_{\lambda}$ of $M \in\mathcal{O}$. $\Theta=\sum_{\zeta\in Q^+}\Theta_\zeta
$ is a well-defined operator on such $M\otimes M$.
\end{lemma}
 \begin{theorem}
\cite{BGH1} For $M, M^{\prime}$ to be finite-dimensional modules of $U_{r,s}(\mathfrak{g})$, the map
 $$R_{M^{\prime},M}=\Theta\circ \widetilde{f}\circ P: M^{\prime} \otimes M \rightarrow M\otimes M^{\prime}$$
 must be an isomorphism of $U_{r,s}(\mathfrak{g})$-module, where $P(m\otimes m^{\prime})=m^{\prime}\otimes m, m\in M, m^{\prime}\in M^{\prime}$.
 \end{theorem}
 \begin{theorem}
\cite{BGH1} For $M,M^{\prime},M^{\prime\prime}$ to be finite-dimensional modules of $U_{r,s}(\mathfrak{g})$, the map must satisfy
 $$\Theta^{f}_{12}\circ\Theta^{f}_{13}\circ\Theta^{f}_{23}=\Theta^{f}_{23}\circ\Theta^{f}_{13}\circ\Theta^{f}_{12}.$$
 Equivalently, we have
 $$R_{12}\circ R_{23}\circ R_{12}=R_{23}\circ R_{12}\circ R_{23}.$$
 \end{theorem}
 \begin{remark}
\rm Denote $\Theta^{f}_{12}=(\Theta_{M,M^{\prime}}\circ\widetilde{f})\otimes1,$ $\Theta^{f}_{23}=1\otimes(\Theta_{M^{\prime},M^{\prime\prime}}\circ\widetilde{f}),$
$\Theta^{f}_{13}=(1\otimes P)\circ(\Theta^{f}\otimes 1)\circ (1\otimes P)$. Also, $\Theta\circ\widetilde{f}$ is the solution of the quantum Yang-Baxter equation.
 \end{remark}

\section{Basic $R$-matrix}
In this section, we consider the $R$-matrix $R$ := $R_{V,V}$ for the vector representation $T_1$ = $T_V$ of the Drinfeld-Jimbo algebra $U=U_{r,s}\mathcal(\mathfrak {so}_{2n})$.
\begin{defi}
\rm The vector representation $T_1$ of the Drinfeld-Jimbo algebra $U$ is the irreducible type $1$ representation with highest weight $\lambda=(1,0,\dots,0)$ with respect to the ordered sequence $\alpha_1,\dots,\alpha_n$ of simple roots.
\end{defi}
Consider a $2n$-dimensional vector space $V$ over $\mathbb{K}$ with basis $\{v_j \mid 1 \le j \le 2n\}$.
We define an action of the generators of $U=U_{r,s}\mathcal(\mathfrak{so}_{2n})$ by specifying their matrices relative to this basis:
\begin{lemma}
Let $E_{kl}$ be the $2n\times 2n$ matrix with $1$ in the $(k,l)$-position and $0$ elsewhere. The vector representation $T_1$ of $U_{r,s}\mathcal(\mathfrak{so}_{2n})$ is described by the following list\\
$(B1)$
\begin{gather*}
T_{1}(e_{i})=E_{i,i+1}-r^{-\frac{1}{2}}s^{-\frac{1}{2}}E_{(i+1)^{\prime},i^{\prime}},\quad T_{1}(e_{n})=E_{n,n+2}-r^{-\frac{1}{2}}s^{-\frac{1}{2}}E_{(n+2)^{\prime},n^{\prime}},\\
T_{1}(f_{i})=E_{i+1,i}-r^{-\frac{1}{2}}s^{-\frac{1}{2}}E_{i^{\prime},(i+1)^{\prime}},\quad T_{1}(f_{n})=E_{n+2,n}-r^{-\frac{1}{2}}s^{-\frac{1}{2}}E_{n^{\prime},(n+2)^{\prime}},
\end{gather*}
$(B2)$
\begin{gather*}
T_{1}(w_{i})=rE_{i,i}{+}sE_{i+1,i+1}
{+}s^{-1}E_{(i+1)^{\prime},(i+1)^{\prime}}{+}r^{-1}E_{i^{\prime},i^{\prime}}{+}\sum_{j\neq \{i,i+1,i^{\prime},(i+1)^{\prime}\}}E_{j,j},\\
T_{1}(w^{\prime}_{i})=sE_{i,i}{+}rE_{i+1,i+1}
{+}r^{-1}E_{(i+1)^{\prime},(i+1)^{\prime}}{+}
s^{-1}E_{i^{\prime},i^{\prime}}{+}\sum_{j\neq \{i,i+1,i^{\prime},(i+1)^{\prime}\}}E_{j,j},
\end{gather*}
$(B3)$
\begin{gather*}
T_{1}(w_{n})=s^{-1}E_{n-1,n-1}+rE_{n,n}+r^{-1}E_{n^{\prime},n^{\prime}}+sE_{(n-1)^{\prime},(n-1)^{\prime}}+r^{-1}s^{-1}E_{1,1}\\
            +r^{-1}s^{-1}\sum_{1\le j\le n-2}E_{j,j}+rs\sum_{1\le j\le n-2}E_{j^{\prime},j^{\prime}},\\
T_{1}(w^{\prime}_{n})=r^{-1}E_{n-1,n-1}+sE_{n,n}+s^{-1}E_{n^{\prime},n^{\prime}}+rE_{(n-1)^{\prime},(n-1)^{\prime}}+r^{-1}s^{-1}E_{1,1}\\
            +r^{-1}s^{-1}\sum_{1\le j\le n-2}E_{j,j}+rs\sum_{1\le j\le n-2}E_{j^{\prime},j^{\prime}},
\end{gather*}
where $1\le i,j\le n-1$, and $i^{\prime}=2n+1-i$.
\end{lemma}
\begin{proof}
By straightforward calculations one checks that the preceding formulas define a weight representation $T_1$ of the algebra $U_{r,s}\mathcal(\mathfrak{so}_{2n})$ on the vector space $\mathbb{C}^{2n}$. For the basis vector $\mathbf{e}_1=(1,0,\cdots,0)$, we easily verify that $T_1(E_j)\mathbf{e}_1=0$, $T_1(w_1)\mathbf{e}_1=r\mathbf{e}_1$, $T_1(w_i)\mathbf{e}_1=\mathbf{e}_1$, $T_1(w_n)\mathbf{e}_1=r^{-1}s^{-1}\mathbf{e}_1$, $T_1(w^{\prime}_1)\mathbf{e}_1=s\mathbf{e}_1$, $T_1(w^{\prime}_i)\mathbf{e}_1=\mathbf{e}_1$, $T_1(w^{\prime}_n)\mathbf{e}_1=rs\mathbf{e}_1$, for $1\le j\le n$, and $2\le i\le n-1$. Hence $T_1$ is the type $1$ representation with highest weight $\lambda=\alpha_1+\alpha_2+\cdots+\frac{1}{2}(\alpha_{n-1}+\alpha_n)=\epsilon_1$. Thus $T_1$ is indeed the vector representation of $U_{r,s}\mathcal(\mathfrak{so}_{2n})$.
\end{proof}

We illustrate through the following lemmas that the module $V\otimes V$ is decomposed into the direct sum of three simple submodules, $S^{o}(V\otimes V), S^{\prime}(V\otimes V)$ and $\Lambda(V\otimes V)$. These modules are defined and proved by the following three lemmas.
\begin{lemma}
The module $S^{o}(V\otimes V)$ generated by $\sum\limits_{i=1}^{2n}a_{i}v_{i^{\prime}}\otimes v_i$ is simple, and $a_i$ satisfies\\
$\text{\rm (i)}$ \,\ $a_i=(rs^{-1})^{\frac{n-i}{2}}$, $1\le i\le n$,\\
$\text{\rm (ii)}$\; $a_i=1$, $i=n+1$,\\
$\text{\rm (iii)}$ $a_i=(rs^{-1})^{\frac{n-i+1}{2}}$, $n+2\le i\le 2n$.
\end{lemma}
\begin{proof}
The operators $e_k,f_k,w_k,w^{\prime}_k$ act on $\sum\limits_{i=1}^{2n}a_{i}v_{i^{\prime}}\otimes v_i$, then the following computations show that $S^{o}(V\otimes V)$ is a simple module:
\begin{flalign*}
e_k.\,\Bigl(\sum\limits_{i=1}^{2n}a_{i}v_{i^{\prime}}\otimes v_i\Bigr)&=(e_k\otimes1+w_k\otimes e_k)\cdot\Bigl(\sum\limits_{i=1}^{2n}a_{i}v_{i^{\prime}}\otimes v_i\Bigr)\\
&=(a_{(k+1)^{\prime}}-a_{k^{\prime}}r^{\frac{1}{2}}s^{-\frac{1}{2}})\,v_k\otimes v_{(k+1)^{\prime}}\\
&\quad +(a_{k+1}s^{-1}-r^{-\frac{1}{2}}s^{-\frac{1}{2}}a_k)\, v_{(k+1)^{\prime}}\otimes v_k\\
&=0,
\end{flalign*}
and
\begin{flalign*}
e_n.\,\Bigl(\sum\limits_{i=1}^{2n}a_{i}v_{i^{\prime}}\otimes v_i\Bigr)&=(e_n\otimes1+w_n\otimes e_n)\cdot\Bigl(\sum\limits_{i=1}^{2n}a_{i}v_{i^{\prime}}\otimes v_i\Bigr)\\
&=(a_{n^{\prime}}r^{\frac{1}{2}}s^{-\frac{1}{2}}-a_{(n+2)^{\prime}})\,v_n\otimes v_{(n+2)^{\prime}}\\
&\quad +(r^{-\frac{1}{2}}s^{-\frac{1}{2}}a_n-a_{(n+2)^{\prime}}s^{-1})\, v_{(n+2)^{\prime}}\otimes v_n\\
&=0.
\end{flalign*}
Moreover,
\begin{flalign*}
w_k.\,\Bigl(\sum\limits_{i=1}^{2n}a_{i}v_{i^{\prime}}\otimes v_i\Bigr)=\sum\limits_{i=1}^{2n}a_{i}v_{i^{\prime}}\otimes v_i,\\
w^{\prime}_k.\,\Bigl(\sum\limits_{i=1}^{2n}a_{i}v_{i^{\prime}}\otimes v_i\Bigr)=\sum\limits_{i=1}^{2n}a_{i}v_{i^{\prime}}\otimes v_i,
\end{flalign*}
and for $1\le k\le n$. By similar calculations as above, we know that the operators $f_k$ act trivially on $\sum\limits_{i=1}^{2n}a_{i}v_{i^{\prime}}\otimes v_i$. So, it is easy to see that $S^{o}(V\otimes V)$ is simple.
\end{proof}
\begin{lemma}
The simple module $S^{\prime}(V\otimes V)$ is defined as follows\\
$\text{\rm (i)}$\;\;\  $v_i\otimes v_i$, $1\le i\le 2n$,\\
$\text{\rm (ii)}$\; $v_i\otimes v_j+sv_j\otimes v_i$, $1\le i\le n$ and $i+1\le j\le n$ or $i^{\prime}+1\le j\le 2n$,\\
$\text{\rm (iii)}$\, $v_i\otimes v_j+r^{-1}v_j\otimes v_i$, $1\le i\le n-1$, $n+1\le j\le 2n-i$ or $n+1\le i\le 2n-1$, $i+1\le j\le 2n$\\
$\text{\rm (iv)}$\, $v_i\otimes v_{i^{\prime}}+r^{-1}sv_{i^{\prime}}\otimes v_i-(r^{-1}s)^{\frac{1}{2}}(v_{(i+1)^{\prime}}\otimes v_{i+1}+v_{i+1}\otimes v_{(i+1)^{\prime}})$, $1\le i\le n-1$,\\
where $v_1\otimes v_1$ is the highest weight vector.
\end{lemma}
\begin{proof}
Operators $e_k,f_k$ act on vectors in (i)-(iv). The following computations show that $S^{\prime}(V\otimes V)$ is a simple module, because for case (i), we have
\begin{flalign*}
e_k.\,(v_i\otimes v_i)=\left\{
\begin{aligned}
&\delta_{i,k+1}(v_k\otimes v_{k+1}+sv_{k+1}\otimes v_k),\\
&-\delta_{i,k^{\prime}}(rs)^{-\frac{1}{2}}(v_{i-1}\otimes v_i+r^{-1}v_i\otimes v_{i-1}), \\
\end{aligned}\right.
\end{flalign*}
or
\begin{flalign*}
f_k.\,(v_i\otimes v_i)=\left\{
\begin{aligned}
&\delta_{i,k}(v_k\otimes v_{k+1}+sv_{k+1}\otimes v_k),\\
&-\delta_{i,(k+1)^{\prime}}(rs)^{-\frac{1}{2}}(v_{i}\otimes v_{i+1}+r^{-1}v_{i+1}\otimes v_{i}), \\
\end{aligned}\right.
\end{flalign*}
where $1\le k\le n-1$, and for $k=n$, we have
\begin{flalign*}
e_n.\,(v_i\otimes v_i)=\left\{
\begin{aligned}
&\delta_{i,n+2}(v_n\otimes v_{n+2}+sv_{n+2}\otimes v_n),\\
&-\delta_{i,n+1}(rs)^{-\frac{1}{2}}(v_{n-1}\otimes v_{n+1}+r^{-1}v_{n+1}\otimes v_{n-1}), \\
\end{aligned}\right.
\end{flalign*}
\begin{flalign*}
f_n.\,(v_i\otimes v_i)=\left\{
\begin{aligned}
&\delta_{i,n}(v_n\otimes v_{n+2}+sv_{n+2}\otimes v_n),\\
&-\delta_{i,n-1}(rs)^{-\frac{1}{2}}(v_{n-1}\otimes v_{n+1}+r^{-1}v_{n+1}\otimes v_{n-1}). \\
\end{aligned}\right.
\end{flalign*}
For case (ii), we have
\begin{equation*}
\begin{split}
e_k.\, &(v_i\otimes v_{j}+sv_{j}\otimes v_i)\\
&=e_k(v_i)\otimes v_j+sw_k(v_j)\otimes e_k(v_i)+se_k(v_j)\otimes v_i+w_k(v_i)\otimes e_k(v_j),
\end{split}
\end{equation*}
so we get
\begin{flalign*}
e_k.\,(v_i\otimes v_{j}+sv_{j}\otimes v_i)=\left\{
\begin{aligned}
&\delta_{i+j,2n+2}\Bigl(v_i\otimes v_{i^{\prime}}+r^{-1}sv_{i^{\prime}}\otimes v_i\\
&\qquad-(r^{-1}s)^{\frac{1}{2}}v_{(i+1)^{\prime}}\otimes v_{i+1}
+v_{i+1}\otimes v_{(i+1)^{\prime}}\Bigr),\\
&\delta_{i,k+1}(v_{i-1}\otimes v_{j}+sv_{j}\otimes v_{i-1}),\\
&\delta_{j,k^{\prime}}(v_{i}\otimes v_{j-1}+sv_{j-1}\otimes v_{i}),\\
&\delta_{i,k}\delta_{k+1,j}(v_{i}\otimes v_{i}).
\end{aligned}\right.
\end{flalign*}
For operators $f_k$, we have
\begin{equation*}
	\begin{split}
f_k.\,(v_i\otimes v_{j}+sv_{j}\otimes v_i)&=v_i\otimes f_k(v_j)+sv_j\otimes f_k(v_i)\\
&\quad +f_k(v_i)\otimes w^{\prime}_k (v_j)+sf_k(v_j)\otimes w^{\prime}_k (v_i),
\end{split}
\end{equation*}
so we have
\begin{flalign*}
f_k.\,(v_i\otimes v_{j}+sv_{j}\otimes v_i)=\left\{
\begin{aligned}
&\delta_{i+j,2n}\Bigl(v_i\otimes v_{i^{\prime}}+r^{-1}sv_{i^{\prime}}\otimes v_i\\
&\qquad-(r^{-1}s)^{\frac{1}{2}}v_{(i+1)^{\prime}}\otimes v_{i+1}+v_{i+1}\otimes v_{(i+1)^{\prime}}\Bigr),\\
&\delta_{j,(k+1)^{\prime}}(v_i\otimes v_{j+1}+sv_{j+1}\otimes v_i),\\
&\delta_{i,k}(v_{i+1}\otimes v_{j}+sv_{j}\otimes v_{i+1}),\\
&\delta_{i,(k+1)^{\prime}}\delta_{j,k^{\prime}}(v_{i+1}\otimes v_{i+1}),\\
\end{aligned}\right.
\end{flalign*}
where $1\le k\le n-1$. For $k=n$, and $n^{\prime}+2\le l\le 2n$, we have
\begin{flalign*}
f_n.\,(v_i\otimes v_{j}+sv_{j}\otimes v_i)=\left\{
\begin{aligned}
&v_{n-1}\otimes v_{(n-1)^{\prime}}+r^{-1}sv_{(n-1)^{\prime}}\otimes v_{n-1}\\
&\qquad-(r^{-1}s)^{\frac{1}{2}}(v_{n}\otimes v_{n^{\prime}}+v_{n^{\prime}}\otimes v_{n}),\\
&\delta_{i,n}(v_{n+2}\otimes v_{l}+r^{-1}v_{l}\otimes v_{n+2}),\\
&\delta_{i,n-1}(v_{n+1}\otimes v_{l}+r^{-1}v_{l}\otimes v_{n+1}),\\
&v_l\otimes v_l.\\
\end{aligned}\right.
\end{flalign*}
For case (iii), we have
\begin{equation*}
	\begin{split}
e_k\cdot(v_i\otimes v_{j}+r^{-1}v_{j}\otimes v_i)&=e_k(v_i)\otimes v_j+r^{-1}w_k(v_j)\otimes e_k(v_i)\\
&\quad+r^{-1}e_k(v_j)\otimes v_i+w_k(v_i)\otimes e_k(v_j),
\end{split}
\end{equation*}
so we get
\begin{flalign*}
e_k.\,(v_i\otimes v_{j}+r^{-1}v_{j}\otimes v_i)=\left\{
\begin{aligned}
&\delta_{i+j,2n+2}\Bigl(v_i\otimes v_{i^{\prime}}+r^{-1}s\,v_{i^{\prime}}\otimes v_i\\
&\qquad-(r^{-1}s)^{\frac{1}{2}}(v_{(i+1)^{\prime}}\otimes v_{i+1}+v_{i+1}\otimes v_{(i+1)^{\prime}})\Bigr),\\
&\delta_{i,k+1}(v_{i-1}\otimes v_{j}+r^{-1}v_{j}\otimes v_{i-1}),\\
&\delta_{j,k^{\prime}}(v_{i}\otimes v_{j-1}+r^{-1}v_{j-1}\otimes v_{i}),\\
&\delta_{i,k}\delta_{k+1,j}(v_{i}\otimes v_{i}).\\
\end{aligned}\right.
\end{flalign*}
For operators $f_k$, we have
\begin{equation*}
	\begin{split}
f_k.\,(v_i\otimes v_{j}+r^{-1}v_{j}\otimes v_i)&=v_i\otimes f_k(v_j)+r^{-1}v_j\otimes f_k(v_i)\\
&\quad+f_k(v_i)\otimes w^{\prime}_k (v_j)+r^{-1}f_k(v_j)\otimes w^{\prime}_k (v_i),
\end{split}
\end{equation*}
so we have
\begin{flalign*}
f_k.\,(v_i\otimes v_{j}+r^{-1}v_{j}\otimes v_i)=\left\{
\begin{aligned}
&\delta_{i+j,2n}\Bigl(v_i\otimes v_{i^{\prime}}+r^{-1}sv_{i^{\prime}}\otimes v_i\\
&\qquad-(r^{-1}s)^{\frac{1}{2}}(v_{(i+1)^{\prime}}\otimes v_{i+1}+v_{i+1}\otimes v_{(i+1)^{\prime}})\Bigr),\\
&\delta_{j,(k+1)^{\prime}}(v_i\otimes v_{j+1}+r^{-1}v_{j+1}\otimes v_i),\\
&\delta_{i,k}(v_{i+1}\otimes v_{j}+r^{-1}v_{j}\otimes v_{i+1}),\\
&\delta_{i,(k+1)^{\prime}}\delta_{j,k^{\prime}}(v_{i+1}\otimes v_{i+1}),\\
\end{aligned}\right.
\end{flalign*}
where $1\le k\le n-1$. For $k=n$, and $n^{\prime}+2\le l\le 2n$, we have
\begin{flalign*}
f_n.\,(v_i\otimes v_{j}+r^{-1}v_{j}\otimes v_i)=\left\{
\begin{aligned}
&v_{n-1}\otimes v_{(n-1)^{\prime}}+r^{-1}sv_{(n-1)^{\prime}}\otimes v_{n-1}\\
&\qquad-(r^{-1}s)^{\frac{1}{2}}(v_{n}\otimes v_{n^{\prime}}+v_{n^{\prime}}\otimes v_{n}),\\
&\delta_{i,n}(v_{n+2}\otimes v_{l}+sv_{l}\otimes v_{n+2}),\\
&\delta_{i,n-1}(v_{n+1}\otimes v_{l}+sv_{l}\otimes v_{n+1}),\\
&v_l\otimes v_l.\\
\end{aligned}\right.
\end{flalign*}
Therefore, it is easy to see that $S^{\prime}(V\otimes V)$ is simple.
\end{proof}
\begin{lemma}
The simple module $\Lambda(V\otimes V)$ is defined as follows:\\
$\text{\rm (i)}$ \,\  $v_i\otimes v_j-rv_j\otimes v_i$, $1\le i\le n$ and $i+1\le j\le n$ or $i^{\prime}+1\le j\le 2n$,\\
$\text{\rm (ii)}$ \ $v_i\otimes v_j-s^{-1}v_j\otimes v_i$, $1\le i\le n-1$, $n+1\le j\le 2n-i$, or $n+1\le i\le 2n-1$, $i+1\le j\le 2n$\\
$\text{\rm (iii)}$  $-(rs)^{-\frac{1}{2}}v_i\otimes v_{i^{\prime}}-s^{-1}v_{(i+1)^{\prime}}\otimes v_{i+1}+r^{-1}v_{i+1}\otimes v_{(i+1)^{\prime}}+(rs)^{-\frac{1}{2}}v_{i^{\prime}}\otimes v_i$, $1\le i\le n-1$,\\
where $v_1\otimes v_2-rv_2\otimes v_1$ is the highest weight vector.
\end{lemma}
\begin{proof}
We can check this lemma by repeating the similar calculations did for Lemma 3.4.
\end{proof}
\begin{lemma}
The decomposition of $U_{r,s}(\mathfrak{so}_{2n})$-module $V\otimes V$ is
$$V\otimes V=S^{o}(V\otimes V)\oplus S^{\prime}(V\otimes V)\oplus\Lambda(V\otimes V).$$
\end{lemma}
\begin{proof}
In \cite{HP}, Hu and Pei have proved that as braided tensor categories, the categories $\mathcal{O}^{r,s}$ of finite-dimensional weight
$U_{r,s}(\mathfrak{g})$-modules (of type $1$) and $\mathcal{O}^{q}$ are monoidally equivalent. Referring to the book by Klimik and Schmuedgen \cite{KS}, $U_{q}(\mathfrak{so}_{2n})$-module $V\otimes V$ is completely reducible and can be decomposed into the direct sum of three simple modules.
\end{proof}
\begin{proposition}
The minimum polynomial of $R=R_{V,V}$ on $V\otimes V$ is $$(t-r^{-\frac{1}{2}}s^{\frac{1}{2}})\cdot(t+r^{\frac{1}{2}}s^{-\frac{1}{2}})\cdot(t-r^{\frac{2n-1}{2}}s^{-\frac{2n-1}{2}}).$$
\end{proposition}
\begin{proof}
It follows from the definition of $R$ that $R(v_1\otimes v_1)=r^{-\frac{1}{2}}s^{\frac{1}{2}}v_1\otimes v_1$ and $R(v_1\otimes v_2-rv_2\otimes v_1)=-r^{\frac{1}{2}}s^{-\frac{1}{2}}(v_1\otimes v_2-rv_2\otimes v_1)$. By the proceeding lemmas, $S^{\prime}(V\otimes V)$ and $\Lambda(V\otimes V)$
are simple, and in fact, $v_1\otimes v_1$ and $v_1\otimes v_2-rv_2\otimes v_1$ are the highest weight vectors. In particular, each is a cyclic module generated by its highest weight vector, respectively, $R(a_1 v_{1^{\prime}}\otimes v_1) = a_1r^{\frac{1}{2}}s^{-\frac{1}{2}}v_1\otimes v_{1^{\prime}}+\cdots$, and $v_1\otimes v_{1^{\prime}}$ only occurs in $R(a_1 v_{1^{\prime}}\otimes v_1)$. So we have the desired result $$R(\sum\limits_{i=1}^{2n}a_{i}v_{i^{\prime}}\otimes v_i)=r^{\frac{2n-1}{2}}s^{-\frac{2n-1}{2}}(\sum\limits_{i=1}^{2n}a_{i}v_{i^{\prime}}\otimes v_i).$$
\end{proof}
\begin{theorem}
The braiding $R$-matrix $R=R_{V,V}$ acts as
\begin{flalign*}
R=&\,r^{-\frac{1}{2}}s^{\frac{1}{2}}
\sum\limits_{i=1}^{2n}E_{ii}{\otimes} E_{ii}{+}r^{\frac{1}{2}}s^{-\frac{1}{2}}\sum\limits_{i=1}^{2n}E_{ii^{\prime}}{\otimes} E_{i^{\prime}i}{+}r^{-\frac{1}{2}}s^{-\frac{1}{2}}
\Bigl\{\sum\limits_{1\le i\le n-1\atop i+1\le j\le n}E_{ij}\otimes E_{ji}\\
&+\sum\limits_{1\le i\le n-1\atop i^{\prime}+1\le j\le 2n}E_{ij}\otimes E_{ji}+\sum\limits_{j=n+2}^{2n}E_{nj}\otimes E_{jn}+\sum\limits_{n+1\le i\le 2n-1\atop i+1\le j\le 2n}E_{ji}\otimes E_{ij}\\
&+\sum\limits_{1\le i\le n-1\atop n+1\le j\le 2n-i}E_{ji}\otimes E_{ij}\Bigr\}+r^{\frac{1}{2}}s^{\frac{1}{2}}\Bigl\{\sum\limits_{1\le i\le n-1\atop i+1\le j\le n}E_{ji}\otimes E_{ij}+\sum\limits_{1\le i\le n-1\atop i^{\prime}+1\le j\le 2n}E_{ji}\otimes E_{ij}\\
&+\sum\limits_{j=n+2}^{2n}E_{jn}\otimes E_{nj}+\sum\limits_{n+1\le i\le 2n-1\atop i+1\le j\le 2n}E_{ij}\otimes E_{ji}+\sum\limits_{1\le i\le n-1\atop n+1\le j\le 2n-i}E_{ij}\otimes E_{ji}\Bigr\}\\
&+(r^{-\frac{1}{2}}s^{\frac{1}{2}}-r^{\frac{1}{2}}s^{-\frac{1}{2}})\Bigl\{\sum\limits_{i<j}E_{jj}\otimes E_{ii}-\sum\limits_{i>j}(r^{-\frac{1}{2}}s^{\frac{1}{2}})^{(\rho_i-\rho_j)}E_{ij^{\prime}}\otimes E_{i^{\prime}j}\Bigr\},
\end{flalign*}
where
$\rho_i=\left\{
\begin{aligned}
&n-i, 1\le i\le n\\
&0, n+1\\
&n-i+1, n+2\le i\le 2n\\
\end{aligned}\right.$.
\begin{proof}
We need to check that the braiding $R$-matrix $R$ acts on $S^{o}(V\otimes V)$ , $S^{\prime}(V\otimes V)$ as multiplication by $r^{\frac{2n-1}{2}}s^{-\frac{2n-1}{2}}$, $r^{-\frac{1}{2}}s^{\frac{1}{2}}$, and on $\Lambda(V\otimes V)$ as multiplication by $-r^{\frac{1}{2}}s^{-\frac{1}{2}}$. By straightforward calculations, one checks that the expression formula of the basic $R$-matrix is correct.
\end{proof}
\end{theorem}
\begin{remark}
\rm Consider the matrix $\hat{R}=P\circ R$, where $P=\sum\limits_{i,j}E_{ij}\otimes E_{ji}$, and $R$ satisfying the braiding relations on the tensor power $V^{\otimes k}$:
\begin{align*}
R_i\circ R_{i+1}\circ R_i&=R_{i+1}\circ R_i\circ R_{i+1},\\
R_i\circ R_j&=R_j\circ R_i,
\end{align*}
where $1\le i< k$, $|i-j|\ge 2,$ $R_i={id_{V}}^{i-1}\otimes R\otimes{id_{V}}^{k-1}$.
\end{remark}

\section{FRT realization of $U_{r,s}(\mathfrak{so}_{2n})$}
In the section, we give an isomorphism between Faddeev-Reshetikhin-Takhtajan and Drinfeld-Jimbo definitions of $U_{r,s}\mathcal(\mathfrak{so}_{2n})$, and the spectral parameter dependent $R(z)$. Let $\mathcal{B}$ (respectively, $\mathcal{B}^{'}$) denote the subalgebra of  $U_{r,s}\mathcal(\mathfrak{so}_{2n})$ generated by $e_{i}$, $w^{\prime\pm1}_{i}$ (respectively, $f_{i}$, $w^{\pm1}_{i}$), $1\le i\le n$.
\begin{defi}
\rm $U(\hat R)$ is an associative algebra with unit. It has generators $l^+_{ij}$, $l^-_{ji}$, $1\le i\le 2n$. Let $L^{\pm}=(l^{\pm}_{ij}), 1\leq i, j\leq 2n+1$, with $l_{ij}^+=l_{ji}^-=0$, and $l_{ii}^-l_{ii}^+=l_{ii}^+l_{ii}^-$ for $1\leq j<i\leq 2n+1$. The defining relations are given in matrix form as follows:
\begin{align}
\hat{R}L^{\pm}_1L^{\pm}_2=L^{\pm}_2L^{\pm}_1\hat{R},\quad \hat{R}L^{+}_1L^{-}_2=L^{-}_2L^{+}_1\hat{R},
\end{align}
where $L^{\pm}_1=L^{\pm}\otimes 1$, $L^{\pm}_2=1\otimes L^{\pm}$.
\end{defi}
Since $L^{\pm}$ are upper and lower triangular, respectively, and the diagonal elements of these matrix are invertible, $L^{\pm}$ have inverse $(L^{\pm})^{-1}$ as matrices with elements in $U(\hat R)$. The relations between $L^{\pm}_1$ and $L^{\pm}_2$ immediately imply the following Theorem.

\begin{theorem}
The mapping $\phi_n$ between $U(\hat R)$ and $U_{r,s}(\mathfrak{so}_{2n})$ is an algebraic homomorphism.
\end{theorem}
\begin{proof}
We check the theorem for the case of $n=4$. Let us consider $L^{\pm}$,
\begin{align*}
L^{+}=\left(\begin{array}{cccc}l^{+}_{11}&l^+_{12}&\cdots&l^+_{18}\\0&l^+_{22}&\ddots&\vdots \\ \vdots&\ddots&\ddots&l^+_{78}\\ 0&\cdots&0&l^+_{88} \end{array}\right)_{8\times8},\quad
L^{-}=\left(\begin{array}{cccc}l^{-}_{11}&0&\cdots&0\\l^-_{21}&l^-_{22}&\ddots&\vdots \\ \vdots&\ddots&\ddots&0\\ l^-_{81}&\cdots&l^+_{87}&l^-_{88} \end{array}\right)_{8\times8}.
\end{align*}
then for the generators $L^{\pm}_1$, $L^{\pm}_2$, $\hat{R}$, we have that
\begin{align*}
L^{+}_1=\left(\begin{array}{cccc}l^{+}_{11}I_8&l^+_{12}I_8&\cdots&l^+_{18}I_8\\0&l^+_{22}I_8&\ddots&\vdots \\ \vdots&\ddots&\ddots&l^+_{78}I_8\\ 0&\cdots&0&l^+_{88}I_8 \end{array}\right)_{64\times64},\quad
L^{-}_1=\left(\begin{array}{cccc}l^{-}_{11}I_8&0&\cdots&0\\l^-_{21}I_8&l^-_{22}I_8&\ddots&\vdots \\ \vdots&\ddots&\ddots&0\\ l^-_{81}I_8&\cdots&l^+_{87}I_8&l^-_{88}I_8 \end{array}\right)_{64\times64},
\end{align*}
\begin{align*}
L^{\pm}_2=\left(\begin{array}{cccc}L^{\pm}&0&\cdots&0\\0&L^{\pm}&\ddots&\vdots \\ \vdots&\ddots&\ddots&0\\ 0&\cdots&0&L^{\pm} \end{array}\right)_{64\times64},\quad
\hat{R}=\left(\begin{array}{cccc} A_{11}&A_{12}&\cdots&A_{18}\\0&A_{22}&\ddots&\vdots \\ \vdots&\ddots&\ddots&A_{78}\\ 0&\cdots&0&A_{88} \end{array}\right)_{64\times64},
\end{align*}
\begin{align*}
A_{11}=\left(\begin{array}{cc}A^{\prime}_{11}&0\\0&QA^{\prime-1}_{11}Q \end{array}\right)_{8\times8},\quad
A^{\prime}_{11}=\left(\begin{array}{cccc} r^{-\frac{1}{2}}s^{\frac{1}{2}}&0&0&0\\0&r^{\frac{1}{2}}s^{\frac{1}{2}}&0&0 \\ 0&0&r^{\frac{1}{2}}s^{\frac{1}{2}}&0\\ 0&0&0&r^{\frac{1}{2}}s^{\frac{1}{2}} \end{array}\right)_{4\times4},
\end{align*}
\begin{align*}
A_{22}=\left(\begin{array}{cc}A^{\prime}_{22}&0\\0&QA^{\prime-1}_{22}Q \end{array}\right)_{8\times8},\quad
A^{\prime}_{22}=\left(\begin{array}{cccc} r^{-\frac{1}{2}}s^{-\frac{1}{2}}&0&0&0\\0&r^{-\frac{1}{2}}s^{\frac{1}{2}}&0&0 \\ 0&0&r^{\frac{1}{2}}s^{\frac{1}{2}}&0\\ 0&0&0&r^{\frac{1}{2}}s^{\frac{1}{2}} \end{array}\right)_{4\times4},
\end{align*}
\begin{align*}
A_{33}=\left(\begin{array}{cc}A^{\prime}_{33}&0\\0&QA^{\prime-1}_{33}Q \end{array}\right)_{8\times8},\quad
A^{\prime}_{33}=\left(\begin{array}{cccc} r^{-\frac{1}{2}}s^{-\frac{1}{2}}&0&0&0\\0&r^{-\frac{1}{2}}s^{-\frac{1}{2}}&0&0 \\ 0&0&r^{-\frac{1}{2}}s^{\frac{1}{2}}&0\\ 0&0&0&r^{\frac{1}{2}}s^{\frac{1}{2}} \end{array}\right)_{4\times4},
\end{align*}
\begin{align*}
A_{44}=\left(\begin{array}{cc}A^{\prime}_{44}&0\\0&QA^{\prime-1}_{44}Q \end{array}\right)_{8\times8},\quad
A^{\prime}_{44}=\left(\begin{array}{cccc} r^{-\frac{1}{2}}s^{-\frac{1}{2}}&0&0&0\\0&r^{-\frac{1}{2}}s^{-\frac{1}{2}}&0&0 \\ 0&0&r^{-\frac{1}{2}}s^{-\frac{1}{2}}&0\\ 0&0&0&r^{-\frac{1}{2}}s^{\frac{1}{2}} \end{array}\right)_{4\times4},
\end{align*}
where $Q=\sum\limits^{4}_{i=1}E_{5-i,i}$, $A_{i^{\prime}i^{\prime}}=A^{-1}_{ii}$, and
$$
A_{ij}=(r^{-\frac{1}{2}}s^{\frac{1}{2}}-r^{\frac{1}{2}}s^{-\frac{1}{2}})\Bigl\{E_{ji}-(r^{-\frac{1}{2}}s^{\frac{1}{2}})^{(\rho_{i^{\prime}}-\rho_{j^{\prime}})}E_{i^{\prime}j^{\prime}}\Bigl\},$$
$1\le i<j\le 8$, $E_{ij}\in M(8,\mathbb{K})$, where the multiplication between matrices $\hat{R}$, $L^{\pm}_1$ and $L^{\pm}_2$ is matrix multiplication. From the equation $\hat{R}L^{+}_1L^{+}_2=L^{+}_2L^{+}_1\hat{R}$, we can derive the following calculations:
\begin{align*}
\hat{R}L^{+}_1L^{+}_2(v_1\otimes v_j)=L^{+}_2L^{+}_1\hat{R}(v_1\otimes v_j)\Rightarrow \begin{cases}
l^{+}_{11}l^{+}_{12}=rl^{+}_{12}l^{+}_{11},\quad l^{+}_{11}l^{+}_{23}=l^{+}_{23}l^{+}_{11},\\
l^{+}_{11}l^{+}_{34}=l^{+}_{34}l^{+}_{11},\quad l^{+}_{11}l^{+}_{35}=r^{-1}s^{-1}l^{+}_{35}l^{+}_{11},
\end{cases}
\end{align*}
where $1\le j\le 7$.
\begin{align*}
\hat{R}L^{+}_1L^{+}_2(v_2\otimes v_j)=L^{+}_2L^{+}_1\hat{R}(v_2\otimes v_j)\Rightarrow \begin{cases}
l^{+}_{22}l^{+}_{12}=sl^{+}_{12}l^{+}_{22}, \quad l^{+}_{22}l^{+}_{23}=rl^{+}_{23}l^{+}_{22},\\
l^{+}_{22}l^{+}_{34}=l^{+}_{34}l^{+}_{22},\quad l^{+}_{22}l^{+}_{35}=(rs)^{-1}l^{+}_{35}l^{+}_{22},\\
l^{+}_{12}l^{+}_{34}=l^{+}_{34}l^{+}_{12},\quad l^{+}_{12}l^{+}_{35}=(rs)^{-1}l^{+}_{35}l^{+}_{12},
\end{cases}
\end{align*}
and we have
\begin{align}
&l^{+}_{12}l^{+}_{23}+(r^{-1}-s^{-1})l^{+}_{22}l^{+}_{l3}=l^{+}_{23}l^{+}_{12},\\
&\qquad l^{+}_{12}l^{+}_{13}=rl^{+}_{13}l^{+}_{12},
\end{align}
where $1\le j\le 8$ $(j\neq7)$.
\begin{align*}
\hat{R}L^{+}_1L^{+}_2(v_3\otimes v_j)=L^{+}_2L^{+}_1\hat{R}(v_3\otimes v_j)\Rightarrow \begin{cases}
l^{+}_{33}l^{+}_{12}=l^{+}_{12}l^{+}_{22}, \quad l^{+}_{33}l^{+}_{23}=sl^{+}_{23}l^{+}_{33},\\
l^{+}_{33}l^{+}_{34}=rl^{+}_{34}l^{+}_{33},\quad l^{+}_{33}l^{+}_{35}=s^{-1}l^{+}_{35}l^{+}_{33},
\end{cases}
\end{align*}
and we have
\begin{align}
&l^{+}_{23}l^{+}_{34}+(r^{-1}-s^{-1})l^{+}_{33}l^{+}_{24}=l^{+}_{34}l^{+}_{23},\\
&l^{+}_{23}l^{+}_{35}+(r^{-1}-s^{-1})l^{+}_{33}l^{+}_{25}=(rs)^{-1}l^{+}_{35}l^{+}_{23},\\
&\qquad l^{+}_{23}l^{+}_{24}=rl^{+}_{24}l^{+}_{23},\\
&\qquad l^{+}_{23}l^{+}_{25}=s^{-1}l^{+}_{25}l^{+}_{23},\\
&\qquad l^{+}_{13}l^{+}_{23}=s^{-1}l^{+}_{23}l^{+}_{13},
\end{align}
where $1\le j\le 8$ $(j\neq6)$.
\begin{align*}
\hat{R}L^{+}_1L^{+}_2(v_4\otimes v_j)=L^{+}_2L^{+}_1\hat{R}(v_4\otimes v_j)\Rightarrow \begin{cases}
l^{+}_{44}l^{+}_{12}=l^{+}_{12}l^{+}_{44}, \quad l^{+}_{44}l^{+}_{23}=l^{+}_{23}l^{+}_{44},\\
l^{+}_{44}l^{+}_{34}=sl^{+}_{34}l^{+}_{44},\quad l^{+}_{44}l^{+}_{35}=rl^{+}_{35}l^{+}_{44},
\end{cases}
\end{align*}
and we have
\begin{align}
&\qquad l^{+}_{24}l^{+}_{34}=s^{-1}l^{+}_{34}l^{+}_{24},\\
&\qquad l^{+}_{25}l^{+}_{35}=s^{-1}l^{+}_{35}l^{+}_{25},
\end{align}
where $1\le j\le 8$ $(j\neq5)$. In particular, we get
\begin{align*}
&\hat{R}L^{+}_1L^{+}_2(v_4\otimes v_5)=L^{+}_2L^{+}_1\hat{R}(v_4\otimes v_5),\\
&\hat{R}L^{+}_1L^{+}_2(v_5\otimes v_4)=L^{+}_2L^{+}_1\hat{R}(v_5\otimes v_4),
\end{align*}
then we obtain
\begin{align}
&l^{+}_{34}l^{+}_{35}+rs^{-1}(r^{-\frac{1}{2}}s^{\frac{1}{2}}-r^{\frac{1}{2}}s^{-\frac{1}{2}})l^{+}_{36}l^{+}_{33}=rs^{-1}l^{+}_{35}l^{+}_{34},\\
&r^{-1}sl^{+}_{35}l^{+}_{34}+(r^{-\frac{1}{2}}s^{\frac{1}{2}}-r^{\frac{1}{2}}s^{-\frac{1}{2}})l^{+}_{36}l^{+}_{33}=l^{+}_{34}l^{+}_{35}.
\end{align}
By (4.2), (4.3) and (4.8), we get
\begin{equation*}
\begin{aligned}
&{l^{+}_{12}}^2 l^{+}_{23}+rs^{-1}l^{+}_{23}{l^{+}_{12}}^2=(rs^{-1}+1)l^{+}_{12}l^{+}_{23}l^{+}_{12}, \\
&{l^{+}_{23}}^2 l^{+}_{12}+r^{-1}sl^{+}_{12}{l^{+}_{23}}^2=(r^{-1}s+1)l^{+}_{23}l^{+}_{12}l^{+}_{23}.
\end{aligned}
\end{equation*}
By equations (4.4), (4.6) and (4.9), we get
\begin{equation*}
\begin{aligned}
&{l^{+}_{23}}^2 l^{+}_{34}+rs^{-1}l^{+}_{34}{l^{+}_{23}}^2=(rs^{-1}+1)l^{+}_{23}l^{+}_{34}l^{+}_{23}, \\
&{l^{+}_{34}}^2 l^{+}_{23}+r^{-1}sl^{+}_{23}{l^{+}_{34}}^2=(r^{-1}s+1)l^{+}_{34}l^{+}_{23}l^{+}_{34}.
\end{aligned}
\end{equation*}
By equations (4.5), (4.7) and (4.10), we get
\begin{equation*}
\begin{aligned}
&s^2{l^{+}_{23}}^2 l^{+}_{35}+(rs)^{-1}l^{+}_{35}{l^{+}_{23}}^2=(r^{-1}s+1)l^{+}_{23}l^{+}_{35}l^{+}_{23}, \\
&(rs)^{-1}{l^{+}_{35}}^2 l^{+}_{23}+s^2l^{+}_{23}{l^{+}_{35}}^2=(r^{-1}s+1)l^{+}_{35}l^{+}_{23}l^{+}_{35}.
\end{aligned}
\end{equation*}
By equations (4.11) and (4.12), we get
\begin{equation*}
l^{+}_{34}l^{+}_{35}=l^{+}_{35}l^{+}_{34}.
\end{equation*}
For the equation $\hat{R}L^{-}_1L^{-}_2=L^{-}_2L^{-}_1\hat{R}$, we can repeat the similar calculation process as above. Then we define a morphism $\phi_4: U(\hat{R})\rightarrow U_{r,s}(\mathfrak{so}_{8}):$
\begin{equation*}
\begin{aligned}
l^+_{11}&\mapsto  (w^{\prime}_1w^{\prime}_2{w^{\prime}_3}^{\frac{1}{2}}{w^{\prime}_4}^{\frac{1}{2}})^{-1},
&l^+_{12}&\mapsto (r-s)e_1l^+_{11},\\
l^+_{22}&\mapsto( w^{\prime}_2{w^{\prime}_3}^{\frac{1}{2}}{w^{\prime}_4}^{\frac{1}{2}})^{-1},    &l^+_{23}&\mapsto (r-s)e_2l^+_{22},\\
l^+_{33}&\mapsto (w^{\prime}_3w^{\prime}_4)^{-\frac{1}{2}},       &l^+_{34}&\mapsto (r-s)e_3l^+_{33},\\
l^+_{44}&\mapsto (w^{\prime-1}_3w^{\prime}_4)^{-\frac{1}{2}},  &l^+_{35}&\mapsto (r-s)e_4l^+_{33},\\
l^-_{11}&\mapsto (w_1w_2(w_3w_4)^{\frac{1}{2}})^{-1}, &l^-_{21}&\mapsto -(r-s)l^-_{11}f_1,\\
l^-_{22}&\mapsto (w_2(w_3w_4)^{\frac{1}{2}})^{-1},    &l^-_{32}&\mapsto -(r-s)l^-_{22}f_2,\\
l^-_{33}&\mapsto (w_3w_4)^{-\frac{1}{2}},       &l^-_{43}&\mapsto -(r-s)l^-_{33}f_3,\\
l^-_{44}&\mapsto (w^{-1}_3w_4)^{-\frac{1}{2}},  &l^-_{53}&\mapsto -(r-s)l^-_{33}f_4,\\
l^+_{i^{\prime}i^{\prime}}&\mapsto (l^+_{ii})^{-1}, &l^-_{i^{\prime}i^{\prime}}&\mapsto (l^-_{ii})^{-1},
\end{aligned}
\end{equation*}
where $1\le i\le 4$. It is obvious that $\phi_4$ still preserves the algebra structure, the relations in $\mathcal{B}$ and $\mathcal{B}^{\prime}$, respectively. Next, we need to ensure that $\phi_4$ preserves the cross relations of $\mathcal{B}$ and $\mathcal{B}^{\prime}$. Considering the equation $\hat{R}L^{+}_1L^{-}_2=L^{-}_2L^{+}_1\hat{R}$, we have
\begin{align*}
\hat{R}L^{+}_1L^{-}_2(v_1\otimes v_j)=L^{-}_2L^{+}_1\hat{R}(v_1\otimes v_j)\Rightarrow \begin{cases}
l^{+}_{11}l^{-}_{21}=r^{-1}l^{-}_{21}l^{+}_{11},\quad l^{+}_{11}l^{-}_{32}=l^{-}_{32}l^{+}_{11},\\
l^{+}_{11}l^{-}_{43}=l^{-}_{43}l^{+}_{11},\quad l^{+}_{11}l^{-}_{53}=rsl^{-}_{53}l^{+}_{11},
\end{cases}
\end{align*}
where $1\le j\le 7$.
\begin{align*}
\hat{R}L^{+}_1L^{-}_2(v_2\otimes v_j)=L^{-}_2L^{+}_1\hat{R}(v_2\otimes v_j)\Rightarrow \begin{cases}
l^{+}_{22}l^{-}_{21}=s^{-1}l^{-}_{21}l^{+}_{22},\quad l^{+}_{22}l^{-}_{32}=r^{-1}l^{-}_{32}l^{+}_{22},\\
l^{+}_{22}l^{-}_{43}=l^{-}_{43}l^{+}_{22},\quad l^{+}_{22}l^{-}_{53}=rsl^{-}_{53}l^{+}_{22},\\
rsl^{+}_{12}l^{-}_{21}-l^{-}_{21}l^{+}_{12}=(s-r)(l^-_{22}l^+_{11}-l^+_{22}l^-_{11}),\\
l^{-}_{32}l^{+}_{12}=rl^{+}_{12}l^{-}_{32},\quad l^{-}_{43}l^{+}_{12}=l^{+}_{12}l^{-}_{43},\\
l^{+}_{12}l^{-}_{53}=rsl^{-}_{53}l^{+}_{12},\quad l^{-}_{11}l^{+}_{12}=sl^{+}_{12}l^{-}_{11},\\
l^{-}_{22}l^{+}_{12}=rl^{+}_{12}l^{-}_{22},\quad l^{-}_{33}l^{+}_{12}=l^{+}_{12}l^{-}_{33},\\
l^{-}_{44}l^{+}_{12}=l^{+}_{12}l^{-}_{44},
\end{cases}
\end{align*}
where $1\le j\le 8$ $(j\neq7)$.
\begin{align*}
\hat{R}L^{+}_1L^{-}_2(v_3\otimes v_j)=L^{-}_2L^{+}_1\hat{R}(v_3\otimes v_j)\Rightarrow \begin{cases}
l^{+}_{33}l^{-}_{21}=l^{-}_{21}l^{+}_{33},\quad l^{+}_{33}l^{-}_{32}=s^{-1}l^{-}_{32}l^{+}_{33},\\
l^{+}_{33}l^{-}_{43}=r^{-1}l^{-}_{43}l^{+}_{33},\quad l^{+}_{33}l^{-}_{53}=sl^{-}_{53}l^{+}_{33},\\
l^{+}_{23}l^{-}_{21}=s^{-1}l^{-}_{21}l^{+}_{23},\\
rsl^{+}_{23}l^{-}_{32}-l^{-}_{32}l^{+}_{23}=(s-r)(l^-_{33}l^+_{22}-l^+_{33}l^-_{22}),\\
l^{+}_{23}l^{-}_{43}=r^{-1}l^{-}_{43}l^{+}_{23},\quad l^{+}_{23}l^{-}_{53}=sl^{-}_{53}l^{+}_{23},\\
l^{+}_{23}l^{-}_{11}=l^{-}_{11}l^{+}_{23},\quad l^{+}_{23}l^{-}_{22}=s^{-1}l^{-}_{22}l^{+}_{23},\\
l^{+}_{23}l^{-}_{33}=r^{-1}l^{-}_{33}l^{+}_{23},\quad l^{+}_{23}l^{-}_{44}=l^{-}_{44}l^{+}_{23},
\end{cases}
\end{align*}
where $1\le j\le 8$ $(j\neq6)$.
\begin{align*}
\hat{R}L^{+}_1L^{-}_2(v_4\otimes v_j)=L^{-}_2L^{+}_1\hat{R}(v_4\otimes v_j)\Rightarrow \begin{cases}
l^{+}_{34}l^{-}_{11}=l^{-}_{11}l^{+}_{34}, \quad l^{+}_{34}l^{-}_{22}=l^{-}_{22}l^{+}_{34}, \\
l^{+}_{34}l^{-}_{33}=s^{-1}l^{-}_{33}l^{+}_{34},\quad l^{+}_{34}l^{-}_{44}=r^{-1}l^{-}_{44}l^{+}_{34},\\
l^{+}_{34}l^{-}_{21}=l^{-}_{21}l^{+}_{34},\quad l^{+}_{34}l^{-}_{32}=s^{-1}l^{-}_{32}l^{+}_{34},\\
rsl^{+}_{34}l^{-}_{43}-l^{-}_{43}l^{+}_{34}=(s-r)(l^-_{44}l^+_{33}-l^+_{44}l^-_{33}),\\
l^{+}_{34}l^{-}_{53}=l^{-}_{53}l^{+}_{34},\\
l^{+}_{44}l^{-}_{21}=l^{-}_{21}l^{+}_{44},\quad l^{+}_{44}l^{-}_{32}=l^{-}_{32}l^{+}_{44},\\
l^{+}_{44}l^{-}_{43}=s^{-1}l^{-}_{43}l^{+}_{44},\quad l^{+}_{44}l^{-}_{53}=sl^{-}_{53}l^{+}_{44},
\end{cases}
\end{align*}
where $1\le j\le 8$ $(j\neq5)$.
\begin{align*}
\hat{R}L^{+}_1L^{-}_2(v_5\otimes v_j)=L^{-}_2L^{+}_1\hat{R}(v_5\otimes v_j)\Rightarrow \begin{cases}
l^{+}_{35}l^{-}_{11}=rsl^{-}_{11}l^{+}_{35}, \quad l^{+}_{35}l^{-}_{22}=rsl^{-}_{22}l^{+}_{35},\\
l^{+}_{35}l^{-}_{33}=rl^{-}_{33}l^{+}_{35},\quad l^{+}_{35}l^{-}_{44}=r^{-1}l^{-}_{44}l^{+}_{35},\\
l^{+}_{35}l^{-}_{21}=rsl^{-}_{21}l^{+}_{35},\quad l^{+}_{35}l^{-}_{32}=rl^{-}_{32}l^{+}_{35},\\
l^{+}_{35}l^{-}_{43}=l^{-}_{43}l^{+}_{35},\\
rsl^{+}_{35}l^{-}_{53}-l^{-}_{53}l^{+}_{35}=(s-r)(l^-_{55}l^+_{33}-l^+_{55}l^-_{33}),
\end{cases}
\end{align*}
where $1\le j\le 8$.
Now we proceed to the case of general $n$, restricting the generating relations (4.1) to $E_{ij}\otimes E_{kl},$ $2\le i,j,k,l\le 2n-1$, by induction, we get all commutation relations except those between $l^{\pm}_{11},l^{+}_{12},l^{-}_{21}$ and $l^{\pm}_{ii},l^{\pm}_{ij}$. Repeating similar computations as above, we have the following relations:
\medskip\noindent
$(B1)$ The $l^{\pm1}_{11},l^{\pm1}_{ii}$ all commute with one another and $l^{\pm1}_{11}(l^{\pm1}_{11})^{-1}=l^{\pm1}_{ii}(l^{\pm1}_{ii})^{-1}=1$.\\
$(B2)$ For $3\le i\le n$, we have
\begin{equation*}
\begin{aligned}
l^{+}_{ii}l^{+}_{12}&=l^{+}_{12}l^+_{ii},\quad  &l^{-}_{ii}l^{+}_{12}&=l^{+}_{12}l^{-}_{ii},\\
l^{+}_{ii}l^{-}_{21}&=l^{-}_{21}l^+_{ii},\quad  &l^{-}_{ii}l^{-}_{21}&=l^{-}_{21}l^{-}_{ii},\\
l^{+}_{22}l^{+}_{12}&=s^{-1}l^{+}_{12}l^+_{22},\quad &l^{+}_{22}l^{-}_{21}&=sl^{-}_{21}l^+_{22},\\
l^{-}_{22}l^{+}_{12}&=r^{-1}l^{+}_{12}l^{-}_{22},\quad &l^{-}_{22}l^{-}_{21}&=rl^{-}_{21}l^{-}_{22}.
\end{aligned}
\end{equation*}
$(B3)$ For $1\le i\le n$, we have
\begin{equation*}
\begin{aligned}
l^{+}_{11}l^{+}_{i,i+1}&=l^{+}_{i,i+1}l^+_{11},\quad& l^{+}_{11}l^{+}_{n-1,n+1}&=(rs)^{-1}l^{+}_{n-1,n+1}l^{+}_{11},\\
l^{-}_{11}l^{-}_{i+1,i}&=l^{-}_{i+1,i}l^-_{11},\quad& l^{-}_{11}l^{-}_{n+1,n-1}&=(rs)^{-1}l^{-}_{n+1,n-1}l^{-}_{11},\\
l^{-}_{11}l^{+}_{i,i+1}&=l^{-}_{i,i+1}l^+_{11},\quad& l^{-}_{11}l^{+}_{n-1,n+1}&=rsl^{+}_{n-1,n+1}l^{-}_{11},\\
l^{+}_{11}l^{-}_{i+1,i}&=l^{-}_{i+1,i}l^+_{11},\quad& l^{+}_{11}l^{-}_{n+1,n-1}&=rsl^{-}_{n+1,n-1}l^{+}_{11}.
\end{aligned}
\end{equation*}
$(B4)$
\begin{equation*}
\begin{aligned}
(l^{+}_{12})^2 l^{+}_{23}+rs^{-1}l^{+}_{23}(l^{+}_{12})^2=(rs^{-1}+1)l^{+}_{12}l^{+}_{23}l^{+}_{12}, \\
(l^{+}_{23})^2 l^{+}_{12}+r^{-1}sl^{+}_{12}(l^{+}_{23})^2=(r^{-1}s+1)l^{+}_{23}l^{+}_{12}l^{+}_{23}, \\
(l^{-}_{21})^2 l^{-}_{32}+rs^{-1}l^{-}_{32}(l^{-}_{21})^2=(rs^{-1}+1)l^{-}_{21}l^{-}_{32}l^{-}_{21}, \\
(l^{-}_{32})^2 l^{-}_{21}+r^{-1}sl^{-}_{21}(l^{-}_{32})^2=(r^{-1}s+1)l^{-}_{32}l^{-}_{21}l^{-}_{32}.
\end{aligned}
\end{equation*}
$(B5)$ For $3\le i\le n-1$, we have
\begin{equation*}
\begin{aligned}
l^{+}_{12}l^{+}_{i,i+1}&=l^{+}_{i,i+1}l^{+}_{12}, \quad  &l^{+}_{12}l^{+}_{n-1,n+1}&=(rs)^{-1}l^{+}_{n-1,n+1}l^{+}_{12},\\
l^{-}_{21}l^{+}_{i,i+1}&=l^{+}_{i,i+1}l^{-}_{21}, \quad  &l^{-}_{21}l^{+}_{n-1,n+1}&=(rs)^{-1}l^{+}_{n-1,n+1}l^{-}_{21},\\
l^{+}_{12}l^{-}_{i+1,i}&=l^{-}_{i+1,i}l^{+}_{12}, \quad  &l^{+}_{12}l^{-}_{n+1,n-1}&=rsl^{-}_{n+1,n-1}l^{+}_{12},\\
l^{-}_{21}l^{-}_{i+1,i}&=l^{-}_{i+1,i}l^{-}_{21}, \quad  &l^{-}_{21}l^{-}_{n+1,n-1}&=(rs)^{-1}l^{-}_{n+1,n-1}l^{-}_{21}.
\end{aligned}
\end{equation*}
We give explicit expressions of the $L$-functionals $l^{\pm}_{ij}$ in terms of the generators of $U_{r,s}(\mathfrak{so}_{2n})$. Define $\phi_n: U(\hat{R})\rightarrow U_{r,s}(\mathfrak{so}_{2n})$ as follows
\begin{equation*}
\begin{aligned}
l^+_{ii}&\mapsto (w^{\prime}_{\beta_{i}})^{-1},    &l^+_{i,i+1}&\mapsto (r-s)e_il^+_{ii},\\
l^-_{ii}&\mapsto (w_{\beta_{i}})^{-1},              &l^-_{i+1,i}&\mapsto -(r-s)l^-_{ii}f_i,\\
l^+_{nn}&\mapsto (w^{\prime}_{\beta_n})^{-1},  &l^+_{n-1,n+1}&\mapsto (r-s)e_nl^+_{n-1,n-1},\\
l^-_{nn}&\mapsto (w_{\beta_n})^{-1},  &l^-_{n+1,n-1}&\mapsto -(r-s)l^-_{n-1,n-1}f_n,\\
l^+_{i^{\prime}i^{\prime}}&\mapsto w^{\prime}_{\beta_{i}},    &l^-_{i^{\prime}i^{\prime}}&\mapsto w_{\beta_{i}},
\end{aligned}
\end{equation*}
where $\beta_i=\alpha_i+\cdots+\alpha_{n-2}+\frac{1}{2}(\alpha_{n-1}+\alpha_n)$, $\beta_n=\frac{1}{2}(\alpha_{n}-\alpha_{n-1})$, $1\le i\le n-1$. By induction, we can prove that $\phi_n$ still preserves the structure of algebra $U_{r,s}(\mathfrak{so}_{2n})$.
\end{proof}
\begin{theorem}
$\phi_n: U(\hat{R})\rightarrow U_{r,s}(\mathfrak{so}_{2n})$ is an algebraic isomorphism.
\end{theorem}
\begin{proof}
It is easy to check that the image of $\phi_n$ contains all generators of $U_{r,s}(\mathfrak{so}_{2n})$. Therefore, $\phi_n$ is surjective.

It remains to show that $\phi_n$ is injective. To this end, we need to construct an algebra homomorphism $\psi_n: U_{r,s}(\mathfrak{so}_{2n})\rightarrow U(\hat{R})$
\begin{equation*}
\begin{aligned}
e_i&\mapsto \frac{1}{r-s}l^+_{i,i+1}(l^+_{ii})^{-1},    &f_i&\mapsto \frac{1}{s-r}(l^-_{ii})^{-1}l^-_{i+1,i},\\
w^{\prime}_i&\mapsto (l^+_{ii})^{-1}l^+_{i+1,i+1}, &w_i&\mapsto (l^-_{ii})^{-1}l^-_{i+1,i+1},\\
e_n&\mapsto \frac{1}{r-s}l^+_{n-1,n+1}(l^+_{n-1,n-1})^{-1},  &f_n&\mapsto \frac{1}{s-r}(l^-_{n-1,n-1})^{-1}l^-_{n+1,n-1},\\
w^{\prime}_n&\mapsto (l^+_{nn})^{-1}l^+_{n-1,n-1}, &w_n&\mapsto (l^-_{nn})^{-1}l^-_{n-1,n-1},
\end{aligned}
\end{equation*}
which satisfies $\psi_n\circ \phi_n = \text{id}$.

To prove that $\psi_n$ still preserves the algebra structure of $U(\hat{R})$ is completely similar to that of Theorem 4.2. Hence, $\phi_n$ is injective. (For a similar proof in the one-parameter setting, one can refer to Section 8.5 of \cite{JLM1}).
\end{proof}
\begin{proposition}
For the braiding $R$-matrix $R=R_{VV}$, the spectral parameter dependent $R(z)$ is given by
\begin{flalign*}
R(z)&=\sum\limits_{i=1}^{2n}E_{ii}\otimes E_{ii}+\frac{rs(z-1)}{rz-s}\Bigl\{\sum\limits_{1\le i\le n-1\atop i+1\le j\le n}E_{ij}\otimes E_{ji}+\sum\limits_{1\le i\le n-1\atop i^{\prime}+1\le j\le 2n}E_{ij}\otimes E_{ji}\\
&\quad+\sum\limits_{j=n+2}^{2n}E_{nj}\otimes E_{jn}+\sum\limits_{n+1\le i\le 2n-1\atop i+1\le j\le 2n}E_{ji}\otimes E_{ij}+\sum\limits_{1\le i\le n-1\atop n+1\le j\le 2n-i}E_{ji}\otimes E_{ij}\Bigr\}\\
&\quad +\frac{z{-}1}{rz{-}s}\Bigl\{\sum\limits_{1\le i\le n-1\atop i+1\le j\le n}E_{ji}{\otimes} E_{ij}+\sum\limits_{1\le i\le n-1\atop i^{\prime}+1\le j\le 2n}E_{ji}{\otimes} E_{ij}+\sum\limits_{j=n+2}^{2n}E_{jn}{\otimes} E_{nj}\\
&\quad+\sum\limits_{n+1\le i\le 2n-1\atop i+1\le j\le 2n}E_{ij}{\otimes} E_{ji}{+}\sum\limits_{1\le i\le n-1\atop n+1\le j\le 2n-i}E_{ij}{\otimes} E_{ji}\Bigr\}{+}\frac{r{-}s}{rz{-}s}\bigl\{z\sum_{\substack{i<j\\i^{\prime}\neq j}}E_{jj}{\otimes} E_{ii}\\
&\quad+\sum_{\substack{i>j\\i\neq j^{\prime}}} E_{jj}\otimes E_{ii}\Bigr\}+\frac{1}{(z-r^{1-n}s^{n-1})(rz-s)}
\sum\limits_{i,j=1}^{2n}d_{ij}(z)E_{ij^{\prime}}\otimes E_{i^{\prime}j},
\end{flalign*}
where
$d_{ij}(z)=\left\{
\begin{aligned}
&(s{-}r)z\Bigl\{(r^{-\frac{1}{2}}s^{\frac{1}{2}})^{\rho_j-\rho_i}
(z{-}1)-\delta_{ij^{\prime}}[z{-}(rs^{-1})^{1-n}]\Bigr\}, \quad i>j,\\
&(s{-}r)\Bigl\{(r^{-\frac{1}{2}}s^{\frac{1}{2}})^{\rho_j-\rho_i+2n-2}(z{-}1)-\delta_{ij^{\prime}}[z{-}(rs^{-1})^{1-n}]\Bigr\}, \quad i<j,\\
&s\Bigl[z-(rs^{-1})^{2-n}\Bigr](z{-}1), \quad i=j.\\
\end{aligned}\right.$
\end{proposition}
\begin{remark}
\rm Consider the $\hat R$-matrix $\hat R(z)=P\circ R(z)$, where $P$ is defined as in Remark 3.9
\begin{flalign*}
\hat R(z)&=\sum\limits_{i=1}^{2n}E_{ii}\otimes E_{ii}+\frac{rs(z-1)}{rz-s}\Bigl\{\sum\limits_{1\le i\le n-1\atop i+1\le j\le n}E_{jj}\otimes E_{ii}+\sum\limits_{1\le i\le n-1\atop i^{\prime}+1\le j\le 2n}E_{jj}\otimes E_{ii}\\
&+\sum\limits_{j=n+2}^{2n}E_{jj}\otimes E_{nn}+\sum\limits_{n+1\le i\le 2n-1\atop i+1\le j\le 2n}E_{ii}\otimes E_{jj}+\sum\limits_{1\le i\le n-1\atop n+1\le j\le 2n-i}E_{ii}\otimes E_{jj}\Bigr\}\\
&+\frac{z-1}{rz-s}\Bigl\{\sum\limits_{1\le i\le n-1\atop i+1\le j\le n}E_{ii}\otimes E_{jj}+\sum\limits_{1\le i\le n-1\atop i^{\prime}+1\le j\le 2n}E_{ii}\otimes E_{jj}+\sum\limits_{j=n+2}^{2n}E_{nn}\otimes E_{jj}\\
&+\sum\limits_{n+1\le i\le 2n-1\atop i+1\le j\le 2n}E_{jj}{\otimes} E_{ii}+\sum\limits_{1\le i\le n-1\atop n+1\le j\le 2n-i}E_{jj}{\otimes} E_{ii}\Bigr\}+\frac{r{-}s}{rz{-}s}\bigl\{z\sum_{\substack{i<j\\i^{\prime}\neq j}}E_{ij}{\otimes} E_{ji}\\
&+\sum_{\substack{i>j\\i\neq j^{\prime}}} E_{ij}\otimes E_{ji}\Bigr\}+\sum\limits_{i,j=1}^{2n}c_{ij}(z)E_{i^{\prime}j^{\prime}}\otimes E_{ij},
\end{flalign*}
where $$c_{ij}(z)=\frac{d_{ij}(z)}{(z-r^{1-n}s^{n-1})(rz-s)}$$.
It is easy to check that $\hat R(z)$ satisfies the quantum Yang-Baxter equation:
\begin{flalign*}
\hat R_{12}(z)\hat R_{13}(zw)\hat R_{23}(w)=\hat R_{23}(w)\hat R_{13}(zw)\hat R_{12}(z),
\end{flalign*}
and the unitary condition:
\begin{align}
 \hat R_{21}(z)\hat R(z^{-1})=\hat R(z^{-1})\hat R_{21}(z)=1.
\end{align}
\end{remark}

\section{The algebra $\mathcal{U}(\hat R)$ and its Gauss decomposition}
\begin{defi}
\rm The algebra $\mathcal{U}(\hat R)$ is an associative algebra with generators $l^{\pm}_{kl}[\mp m]$ $(m\in\mathbb{Z}_+\setminus \{\,0\,\})$, and $l^+_{kl}[0]$= $l^-_{lk}[0]$=0, $1\le l\le k\le n$ and the central element $c$ via $r^{\frac{c}{2}}$ or $s^{\frac{c}{2}}$. Let $l^{\pm}_{ij}(z)=\sum\limits_{m=0}^{\infty}l^{\pm}_{ij}[\mp m]z^{\pm m}$, and $L^{\pm}(z)=\sum\limits_{i,j=1}^{n}E_{ij}\otimes l^{\pm}_{ij}(z)$. Then the relations are given by the following matrix equations on End$(V^{\otimes 2})\otimes \mathcal{U}(\hat R)$:
\begin{align}
&l^+_{ii}[0],\ l^-_{ii}[0]\ \text{are invertible and}\ l^+_{ii}[0]\,l^-_{ii}[0]=l^-_{ii}[0]\, l^+_{ii}[0],\\
&\qquad\hat R(\frac{z}{w})L^{\pm}_{1}(z)L^{\pm}_{2}(w)=L^{\pm}_{2}(w)L^{\pm}_{1}(z)\hat R(\frac{z}{w}),\\
&\qquad\hat R(\frac{z_+}{w_-})L^{+}_{1}(z)L^{-}_{2}(w)=L^{-}_{2}(w)L^{+}_{1}(z)\hat R(\frac{z_-}{w_+}),
\end{align}
where $z_+$ = $zr^{\frac{c}{2}}$ and $z_-$ = $zs^{\frac{c}{2}}$. Here (5.2) is expanded in the direction of either $\frac{z}{w}$ or $\frac{w}{z}$, and (5.3) is expanded in the direction of $\frac{z}{w}$.
\end{defi}
\begin{remark}
\rm From Eq (5.3) and the unitary condition of $\hat R$-matrix (4.13), we have
\begin{align}
\hat R(\frac{z_\pm}{w_\mp})L^{\pm}_{1}(z)L^{\mp}_{2}(w)=L^{\mp}_{2}(w)L^{\pm}_{1}(z)\hat R(\frac{z_\mp}{w_\pm}).
\end{align}
So the relations of generating series (5.2), (5.3) are equivalent to
\begin{align}
&L^{\pm}_{1}(z)^{-1}L^{\pm}_{2}(w)^{-1}\hat R(\frac{z}{w})=\hat R(\frac{z}{w})L^{\pm}_{2}(w)^{-1}L^{\pm}_{1}(z)^{-1},\\
&L^{\pm}_{1}(z)^{-1}L^{\mp}_{2}(w)^{-1}\hat R(\frac{z_\pm}{w_\mp})=\hat R(\frac{z_\mp}{w_\pm})L^{\mp}_{2}(w)^{-1}L^{\pm}_{1}(z)^{-1}.
\end{align}
They are also equivalent to
\begin{align}
&L^{\pm}_{2}(w)^{-1}\hat R(\frac{z}{w})L^{\pm}_{1}(z)=L^{\pm}_{1}(z)\hat R(\frac{z}{w})L^{\pm}_{2}(w)^{-1},\\
&L^{\mp}_{2}(w)^{-1}\hat R(\frac{z_\pm}{w_\mp})L^{\pm}_{1}(z)=L^{\pm}_{1}(z)\hat R(\frac{z_\mp}{w_\pm})L^{\mp}_{2}(w)^{-1}.
\end{align}
\end{remark}
\begin{remark}
\rm Here we present the specific matrix expression formulas for (5.2) and (5.3), and reveal the differences between type $D^{(1)}_n$ and type $A^{(1)}_n$. For  $D^{(1)}_n$, write
\begin{align*}
L^{\pm}(z)=\left(\begin{array}{cccc}l^{\pm}_{11}(z)&l^{\pm}_{12}(z)&\cdots&l^{\pm}_{1,2n}(z)\\l^{\pm}_{21}(z)&l^{\pm}_{22}(z)&\ddots&\vdots \\ \vdots&\ddots&\ddots&l^{\pm}_{2n-1,2n}(z)\\ l^{\pm}_{2n,1}(z)&\cdots&l^{\pm}_{2n,2n-1}(z)&l^{\pm}_{2n,2n}(z) \end{array}\right)_{2n\times2n},
\end{align*}
then for the generators $L^{\pm}_1(z)$, $L^{\pm}_2(z)$, $\hat{R}(z)$, we have that
\begin{align*}
L^{\pm}_1(z)=\left(\begin{array}{ccc}l^{\pm}_{11}(z)I_{2n}&\cdots&l^{\pm}_{1,2n}(z)I_{2n}\\ \vdots&\cdots&\vdots \\l^{\pm}_{2n,1}(z)I_{2n}&\cdots&l^{\pm}_{2n,2n}(z)I_{2n} \end{array}\right)_{4n^2\times4n^2},
\end{align*}
\begin{align*}
L^{\pm}_2(z)=\left(\begin{array}{cccc}L^{\pm}(z)&0&\cdots&0\\0&L^{\pm}(z)&\ddots&\vdots \\ \vdots&\ddots&\ddots&0\\ 0&\cdots&0&L^{\pm}(z) \end{array}\right)_{4n^2\times4n^2},
\end{align*}
\begin{align*}
\hat{R}(z)=\left(\begin{array}{ccc} B_{11}(z)&\cdots&B_{1,2n}(z)\\ \vdots &\cdots&\vdots \\ B_{2n,1}(z)&\cdots&B_{2n,2n}(z) \end{array}\right)_{4n^2\times4n^2},
\end{align*}
\begin{align*}
B_{ll}(z)=\left(\begin{array}{cccc} a_{l1}(z)&0&\cdots&0\\0&a_{l2}(z)&\ddots&\vdots \\ \vdots&\ddots&\ddots&0\\ 0&\cdots&0&a_{l,2n}(z) \end{array}\right)_{2n\times2n},
\end{align*}
where $\bullet$ $B_{ll}(z)$ is a diagonal matrix, and $a_{lj}$ is the coefficient of element $E_{ll}\otimes E_{jj}$ in $\hat{R}(z)$.\\
$\bullet$ $B_{ij}(z)=b_{ij}(z)E_{ji}+c_{i^{\prime}j^{\prime}}(z)E_{i^{\prime}j^{\prime}}$ , where $b_{ij}(z)$ is the coefficient of element $E_{ij}\otimes E_{ji}$ in $\hat{R}(z)$, and $c_{ij}(z)$ is the coefficient of element $E_{i^{\prime}j^{\prime}}\otimes E_{ij}$ in $\hat{R}(z)$.

The multiplication between matrices $\hat{R}(\frac{z}{w})$, $L^{\pm}_1(z)$, $L^{\pm}_2(w)$ is matrix multiplication. From  equation $$\hat{R}(\frac{z}{w})L^{\pm}_1(z)L^{\pm}_2(w)=L^{\pm}_2(w)L^{\pm}_1(z)\hat{R}(\frac{z}{w}),$$ we can derive the following calculation:
\begin{align*}
\hat{R}(\frac{z}{w})L^{\pm}_1(z)L^{\pm}_2(w)=\left(\begin{array}{ccc}M_{11}&\cdots&M_{1,2n}\\ \vdots&\cdots&\vdots \\M_{2n,1}&\cdots&M_{2n,2n} \end{array}\right)_{4n^2\times4n^2},\quad M_{ij}\in M(2n,\mathbb{K}),
\end{align*}
\begin{align*}
L^{\pm}_2(w)L^{\pm}_1(z)\hat{R}(\frac{z}{w})=\left(\begin{array}{ccc}M^{\prime}_{11}&\cdots&M^{\prime}_{1,2n}\\ \vdots&\cdots&\vdots \\M^{\prime}_{2n,1}&\cdots&M^{\prime}_{2n,2n} \end{array}\right)_{4n^2\times4n^2},\quad M^{\prime}_{ij}\in M(2n,\mathbb{K}).
\end{align*}
We only give two types of matrix expressions that will be used later. Taking $M_{ij}=M^{\prime}_{ij}$, where $1\le i,j\le n$, consider $M_{ij}$
\begin{equation}
\begin{array} {cc}
\left(\begin{array}{ccccccc}
a_{i1}(\frac{z}{w})l^{\pm}_{ij}(z) &\quad   &b_{i1}(\frac{z}{w})l^{\pm}_{1j}(z) &\quad &\quad &\quad \\
\quad &\ddots  &\vdots                 &\quad    &\quad\\
\quad &\quad   &a_{ii}(\frac{z}{w})l^{\pm}_{ij}(z)  &\quad    &\quad\\
\quad  &\quad   &\vdots                 &\ddots    &\quad\\
c_{i^{\prime}1}(\frac{z}{w})l^{\pm}_{1^{\prime}j}(z) &\cdots &c_{i^{\prime}i}(\frac{z}{w})l^{\pm}_{i^{\prime}j}(z)  &\cdots   & c_{i^{\prime}i^{\prime}}(\frac{z}{w})l^{\pm}_{ij}(z) &\cdots &c_{i^{\prime}1^{\prime}}(\frac{z}{w})l^{\pm}_{1j}(z)\\
\quad &\quad &\vdots &\quad &\quad &\ddots\\
\quad &\quad &b_{i1^{\prime}}(\frac{z}{w})l^{\pm}_{1^{\prime}j}(z) &\quad &\quad&\quad &a_{i1^{\prime}}(\frac{z}{w})l^{\pm}_{ij}(z)
         \end{array}\right)
\end{array}L^{\pm}(w),
\end{equation}
and $1\le i\le n, 1+n\le j$, $M_{ij}$
\begin{equation}
\begin{array} {cc}
\left(\begin{array}{ccccccc}
a_{i1}(\frac{z}{w})l^{\pm}_{ij}(z) &\quad   &\quad &\quad &b_{i1}(\frac{z}{w})l^{\pm}_{1j}(z) &\quad \\
\quad &\ddots  &\quad                 &\quad    &\vdots &\quad\\
c_{i^{\prime}1}(\frac{z}{w})l^{\pm}_{1^{\prime}j}(z) &\cdots &c_{i^{\prime}i}(\frac{z}{w})l^{\pm}_{ij}(z) &\cdots &c_{i^{\prime}i^{\prime}}(\frac{z}{w})l^{\pm}_{i^{\prime}j}(z) &\cdots &c_{i^{\prime}1^{\prime}}(\frac{z}{w})l^{\pm}_{1j}(z)\\
\quad  &\quad   &\quad                &\ddots    &\vdots &\quad\\
\quad &\quad &\quad  &\quad   &a_{ii}(\frac{z}{w})l^{\pm}_{ij}(z) &\quad &\quad\\
\quad &\quad &\quad &\quad &\vdots &\ddots\\
\quad &\quad &\quad &\quad &b_{i1^{\prime}}(\frac{z}{w})l^{\pm}_{1^{\prime}j}(z) &\quad &a_{i1^{\prime}}(\frac{z}{w})l^{\pm}_{ij}(z)
         \end{array}\right)
\end{array}L^{\pm}(w),
\end{equation}
where the elements in the $i^{\prime}$-th row except for the element at position $(i^{\prime}, i^{\prime})$ are all zero for type $A_n^{(1)}$. Consider $M^{\prime}_{ij}$, for $1\le i,j\le n$
\begin{equation}
L^{\pm}(w)
\begin{array} {cc}
\left(\begin{array}{ccccccc}
a_{j1}(\frac{z}{w})l^{\pm}_{ij}(z)  &\quad &\quad &\quad  &c_{1j^{\prime}}(\frac{z}{w})l^{\pm}_{i1^{\prime}}(z) &\quad &\quad\\
\quad &\ddots &\quad &\quad &\vdots &\quad &\quad\\
b_{1j}(\frac{z}{w})l^{\pm}_{i1}(z) &\cdots &a_{jj}(\frac{z}{w})l^{\pm}_{ij}(z) &\cdots &c_{jj^{\prime}}(\frac{z}{w})l^{\pm}_{ij^{\prime}}(z) &\cdots &b_{1^{\prime}j}(\frac{z}{w})l^{\pm}_{i1^{\prime}}(z)\\
\quad &\quad &\quad &\ddots &\vdots &\quad   &\quad\\
\quad &\quad &\quad &\quad  &c_{j^{\prime}j^{\prime}}(\frac{z}{w})l^{\pm}_{ij}(z)  &\quad  &\quad\\
\quad &\quad &\quad &\quad  &\vdots &\ddots &\quad\\
\quad &\quad &\quad &\quad  &c_{1^{\prime}j^{\prime}}(\frac{z}{w})l^{\pm}_{i1}(z) &\quad &a_{j1^{\prime}}(\frac{z}{w})l^{\pm}_{ij}(z)
         \end{array}\right)
\end{array},
\end{equation}
moreover, $1\le i\le n, 1+n\le j$,
\begin{equation}
L^{\pm}(w)
\begin{array} {cc}
\left(\begin{array}{ccccccc}
a_{j1}(\frac{z}{w})l^{\pm}_{ij}(z) &\quad &c_{1j^{\prime}}(\frac{z}{w})l^{\pm}_{i1^{\prime}}(z) &\quad &\quad &\quad &\quad\\
\quad &\ddots &\vdots &\quad &\quad &\quad &\quad\\
\quad &\quad &c_{jj^{\prime}}(\frac{z}{w})l^{\pm}_{ij}(z) &\quad &\quad &\quad &\quad\\
\quad &\quad &\vdots &\ddots &\quad &\quad   &\quad\\
b_{1j}(\frac{z}{w})l^{\pm}_{i1}(z)&\cdots &c_{j^{\prime}j^{\prime}}(\frac{z}{w})l^{\pm}_{ij^{\prime}}(z) &\cdots &a_{jj}(\frac{z}{w})l^{\pm}_{ij}(z) &\cdots &b_{1^{\prime}j}(\frac{z}{w})l^{\pm}_{i1^{\prime}}(z)\\
\quad &\quad &\vdots &\quad  &\quad &\ddots &\quad\\
\quad &\quad &c_{1^{\prime}j^{\prime}}(\frac{z}{w})l^{\pm}_{i1}(z) &\quad  &\quad &\quad &a_{j1^{\prime}}(\frac{z}{w})l^{\pm}_{ij}(z)
         \end{array}\right)
\end{array},
\end{equation}
where the elements in the $j^{\prime}$-th column except for the element at position $(j^{\prime}, j^{\prime})$ are all zero for type $A_n^{(1)}$.
\end{remark}
\begin{defi}
\rm Let $X$ = $(x_{ij})^{n}_{i,j=1}$ be a sequence matrix over a ring with identity. Denote by $X^{ij}$ the submatrix obtained from $X$ by deleting the $i$-th row and $j$-th column. Suppose that the matrix $X^{ij}$ is invertible. The $(i,j)$-th quasi-determinant $|X|_{ij}$ of $X$ is defined by
\end{defi}
\begin{equation}
|X|_{ij}=\left|\begin{array}{rrrrr}x_{11}&\cdots&x_{1j}&\cdots&x_{1n}\\ \qquad &\cdots&\qquad &\cdots&\qquad \\x_{i1}&\cdots&\boxed{x_{ij}}&\cdots&x_{in}\\\qquad &\cdots&\qquad &\cdots&\qquad \\x_{n1}&\cdots&x_{nj}&\cdots&x_{nn}\end{array}\right|=x_{ij}-r^j_i(X^{ij})^{-1}c^i_j,
\end{equation}
where $r^j_i$ is the row matrix obtained from the $i$-th row of $X$ by deleting the element $x_{ij}$, and $c^i_j$ is the column matrix obtained from the $j$-th column of $X$ by deleting the element $x_{ij}$.
\begin{proposition}
$L^{\pm}(z)$ have the following unique decomposition
\begin{align}
&L^{\pm}(z)=F^{\pm}(z)K^{\pm}(z)E^{\pm}(z),
\end{align}
by applying the Gauss decomposition to $L^{\pm}(z)$, where we introduce matrices with $N\times N$, and $N=2n$,
\begin{align}
F^{\pm}(z)=\left(\begin{array}{cccc}1\quad&\quad&\quad&0\\
	f^{\pm}_{21}(z)&\ddots \\ \vdots&\ddots&\ddots\\ f^{\pm}_{N1}(z) &\cdots&f^{\pm}_{N,N-1}(z)&1 \end{array}\right),
\end{align}
\begin{align}
E^{\pm}(z)=\left(\begin{array}{cccc}1&e^{\pm}_{12}(z)&\cdots&e^{\pm}_{1N}(z)\\ \quad&\ddots&\ddots&\vdots\\ \quad&\quad&\ddots&e^{\pm}_{N-1,N}(z)\\ 0&\quad&\quad&1\end{array}\right),
\end{align}
and
\begin{align}
K^{\pm}(z)=\left(\begin{array}{cccc}k^{\pm}_1(z)&\quad&\quad&0\\ \quad &\ddots&\quad&\quad \\ \quad&\quad&\ddots\quad\\ 0&\quad&\quad&k^{\pm}_N(z)\end{array}\right).
\end{align}
Their entries are found by the quasi-determinant formulas
\begin{equation}
k^{\pm}_{m}(z)=\left|\begin{array}{cccc}l^{\pm}_{11}(z)&\cdots&l^{\pm}_{1,m-1}(z)&l^{\pm}_{1m}(z)\\ \vdots &\quad&\vdots &\vdots \\l^{\pm}_{m1}(z)&\cdots&l^{\pm}_{m,m-1}(z)&\boxed{l^{\pm}_{mm}(z)}\end{array}\right|
\end{equation}
for $1\le m\le 2n$, $k^{\pm}_{m}(z)=\sum\limits_{t\in\mathbb{Z}_+}k^{\pm}_{m}(\mp t)z^{\pm t}$.
\begin{equation}
e^{\pm}_{ij}(z)=k^{\pm}_i(z)^{-1}\left|\begin{array}{cccc}l^{\pm}_{11}(z)&\cdots&l^{\pm}_{1,i-1}(z)&l^{\pm}_{1j}(z)\\ \vdots &\quad&\vdots &\vdots \\l^{\pm}_{i1}(z)&\cdots&l^{\pm}_{i,i-1}(z)&\boxed{l^{\pm}_{ij}(z)}\end{array}\right|
\end{equation}
for $1\le i<j\le 2n$, $e^{\pm}_{ij}(z)=\sum\limits_{m\in\mathbb{Z}_+}e^{\pm}_{ij}(\mp m)z^{\pm m}$.
\begin{equation}
f^{\pm}_{ji}(z)=\left|\begin{array}{cccc}l^{\pm}_{11}(z)&\cdots&l^{\pm}_{1,i-1}(z)&l^{\pm}_{1i}(z)\\ \vdots &\quad&\vdots &\vdots \\l^{\pm}_{j1}(z)&\cdots&l^{\pm}_{j,i-1}(z)&\boxed{l^{\pm}_{ji}(z)}\end{array}\right|
k^{\pm}_i(z)^{-1}
\end{equation}
for $1\le i<j\le 2n$, $f^{\pm}_{ji}(z)=\sum\limits_{m\in\mathbb{Z}_+}f^{\pm}_{ji}(\mp m)z^{\pm m}$.
\end{proposition}

\section{Drinfeld realization of $U_{r,s}\mathcal(\widehat {\mathfrak{so}_{2n}})$}
In this section, we will study the commutation relations between the Gaussian generators and give the Drinfeld realization of
$U_{r,s}(\widehat{\mathfrak{so}_{2n}})$.
\begin{theorem}
In the algebra $\mathcal{U}(\hat R)$, we have
\begin{equation*}
\begin{aligned}
X^{+}_i(z)&=e^{+}_{i,i+1}(z_+){-}e^{-}_{i,i+1}(z_-),\ &X^{+}_n(z)&=e^{+}_{n-1,n+1}(z_+){-}e^{-}_{n-1,n+1}(z_-),\\
X^{-}_i(z)&=f^{+}_{i+1,i}(z_-)-f^{-}_{i+1,i}(z_+),\ &X^{-}_n(z)&=f^{+}_{n+1,n-1}(z_-)-f^{-}_{n+1,n-1}(z_+).
\end{aligned}
\end{equation*}
For the generators $\bigl\{k^{\pm}_i(z), X^{\pm}_j(z), k^{\pm}_{n+1}(z), X^{\pm}_n(z)\mid 1\le i\le n, 1\le j\le n-1\bigr\}$, the relations between $k^{\pm}_i(z)$ and $X^{\pm}_j(z)$ are the same as those in $U_{r,s}(\widehat{\mathfrak{gl}_n})$. The other relations are
those involving $k^{\pm}_i(z)$ and $k^{\pm}_{n+1}(w)$, that is,
\begin{align*}
k^{\pm}_{i}(z)k^{\pm}_{n+1}(w)&=k^{\pm}_{n+1}(w)k^{\pm}_{i}(z),\\
k^{+}_{n+1}(z)k^{-}_{n+1}(w)&=k^{-}_{n+1}(w)k^{+}_{n+1}(z),\\
k^{\mp}_{i}(z)k^{\pm}_{n+1}(w)\frac{z_{\mp}-w_{\pm}}{rz_{\mp}-sw_{\pm}}&=\frac{z_{\pm}-w_{\mp}}{rz_{\pm}-sw_{\mp}}k^{\pm}_{n+1}(w)k^{\mp}_{i}(z).
\end{align*}
The relations involving $k^{\pm}_t(w)\ (1\le t\le n+1)$ and $X^{\pm}_{n}(z)$ are
\begin{align*}
k^{\pm}_l(w)X^{\pm}_{n}(z)&=rsX^{\pm}_{n}(z)k^{\pm}_l(w),\\
rsk^{\pm}_l(w)X^{\mp}_{n}(z)&=X^{\mp}_{n}(z)k^{\pm}_l(w),\quad 1\le l\le n-2,\\
(rw-sz_{\pm})k^{\pm}_{n-1}(w)X^{+}_n(z)&=rs(w-z_{\pm})X^{+}_n(z)k^{\pm}_{n-1}(w),\\
rs(w-z_{\mp})k^{\pm}_{n-1}(w)X^{-}_n(z)&=(rw-sz_{\mp})X^{-}_n(z)k^{\pm}_{n-1}(w),\\
k^{\pm}_n(w)X^{+}_n(z)&=\frac{wr-sz_{\pm}}{w-z_{\pm}}X^{+}_n(z)k^{\pm}_n(w),\\
X^{-}_n(z)k^{\pm}_n(w)&=\frac{wr-sz_{\mp}}{w-z_{\mp}}k^{\pm}_n(w)X^{-}_n(z),\\
k^{\pm}_{n+1}(w)X^{+}_n(z)&=\frac{rs(w-z_{\pm})}{sw-rz_{\pm}}X^{+}_n(z)k^{\pm}_{n+1}(w),\\
X^{-}_n(z)k^{\pm}_{n+1}(w)&=\frac{rs(w-z_{\mp})}{sw-rz_{\mp}}k^{\pm}_{n+1}(w)X^{-}_n(z).
\end{align*}
The relations involving $k^{\pm}_{n+1}(z)$ and $X^{\pm}_{t}(w)$ $(1\le t\le n-1)$ are
\begin{align*}
k^{\pm}_{n+1}(z)X^{\mp}_l(w)&=X^{\mp}_l(w)k^{\pm}_{n+1}(z),\\
k^{\pm}_{n+1}(z)X^{\pm}_l(w)&=X^{\pm}_l(w)k^{\pm}_{n+1}(z),\quad 1\le l\le n-2,\\
k^{\pm}_{n+1}(w)X^{+}_{n-1}(z)&=\frac{rs(z_{\pm}-w)}{z_{\pm}s-rw}X^{+}_{n-1}(z)k^{\pm}_{n+1}(w),\\
X^{-}_{n-1}(z)k^{\pm}_{n+1}(w)&=\frac{rs(z_{\mp}-w)}{z_{\mp}s-rw}k^{\pm}_{n+1}(w)X^{-}_{n-1}(z).
\end{align*}
For the relations involving $X^{\pm}_{n}(z)$ and $X^{\pm}_{t}(z) \ (1\le t\le n)$, we have
\begin{align*}
X^{\pm}_n(w)X^{\pm}_l(z)&=X^{\pm}_l(z)X^{\pm}_n(w),\\
X^{\pm}_n(w)X^{\mp}_l(z)&=X^{\mp}_l(z)X^{\pm}_n(w),\quad 1\le l\le n-3,\\
X^{+}_n(w)X^{+}_{n-2}(z)&=\frac{rz-sw}{z-w}X^{+}_{n-2}(z)X^{+}_n(w),\\
X^{-}_n(w)X^{-}_{n-2}(z)&=\frac{z-w}{rz-sw}X^{-}_{n-2}(z)X^{-}_n(w),\\
X^{\pm}_n(w)X^{\mp}_{n-2}(z)&=X^{\mp}_{n-2}(z)X^{\pm}_n(w),\\
X^{\pm}_{n}(w)X^{\pm}_{n-1}(z)&=(rs)^{\pm 1}X^{\pm}_{n-1}(z)X^{\pm}_n(w),\\
X^{\pm}_{n-1}(w)X^{\mp}_n(z)&=X^{\mp}_{n}(z)X^{\pm}_{n-1}(w),\\
X^{+}_n(z)X^{+}_n(w)&=\frac{zr-ws}{zs-rw}X^{+}_n(w)X^{+}_n(z),\\
X^{-}_n(z)X^{-}_n(w)&=\frac{zs-wr}{zr-sw}X^{-}_n(w)X^{-}_n(z),
\\
\bigl[\,X^{+}_n(z),X^{-}_t(w)\,\bigr]&=(s^{-1}-r^{-1})\delta_{nt}\Bigl\{\delta\Bigl(\frac{z_{-}}{w_+}\Bigr)k^{-}_{n+1}(w_+)k^{-}_n(w_+)^{-1}\\
&\qquad\qquad\qquad\qquad-\delta\Bigl(\frac{z_{+}}{w_-}\Bigr)k^{+}_{n+1}(z_+)k^{+}_n(z_+)^{-1}\Bigr\},
\end{align*}
and the following $(r,s)$-Serre relations hold in $\mathcal{U}(\hat R)$:
\begin{align}
&\Bigl\{X^{-}_{n-2}(z_1)X^{-}_{n-2}(z_2)X^{-}_{n}(w)-(r+s)X^{-}_{n-2}(z_1)X^{-}_{n}(w)X^{-}_{n-2}(z_2)\nonumber\\
&\qquad+rsX^{-}_{n}(w)X^{-}_{n-2}(z_1)X^{-}_{n-2}(z_2)\Bigr\}+\Bigl\{z_1 \leftrightarrow z_2\Bigr\}=0,
\end{align}
\begin{align}
&\Bigl\{X^{+}_{n}(z_1)X^{+}_{n}(z_2)X^{+}_{n-2}(w)-(r+s)X^{+}_{n}(z_1)X^{+}_{n-2}(w)X^{+}_{n}(z_2)\nonumber\\
&\quad+rsX^{+}_{n-2}(w)X^{+}_{n}(z_1)X^{+}_{n}(z_2)\Bigr\}+\Bigl\{z_1 \leftrightarrow z_2\Bigr\}=0,\\
&\Bigl\{rsX^{+}_{n-2}(z_1)X^{+}_{n-2}(z_2)X^{+}_{n}(w)-(r+s)X^{+}_{n-2}(z_1)X^{+}_{n}(w)X^{+}_{n-2}(z_2)\nonumber\\
&\quad+X^{+}_{n}(w)X^{+}_{n-2}(z_1)X^{+}_{n-2}(z_2)\Bigr\}+\Bigl\{z_1 \leftrightarrow z_2\Bigr\}=0,\\
&\Bigl\{rsX^{-}_{n}(z_1)X^{-}_{n}(z_2)X^{-}_{n-2}(w)-(r+s)X^{-}_{n}(z_1)X^{-}_{n-2}(w)X^{-}_{n}(z_2)\nonumber\\
&\quad+X^{-}_{n-2}(w)X^{-}_{n}(z_1)X^{-}_{n}(z_2)\Bigr\}+\Bigl\{z_1 \leftrightarrow z_2\Bigr\}=0,
\end{align}
where the formal delta function $\delta(z)=\sum\limits_{n\in\mathbb{Z}}z^{n}$.
\begin{proof}
The proof is based on the induction on $n$. We consider first the case $n=4$,
\begin{align*}
L^{\pm}(z)=\left(\begin{array}{cccc}l^{\pm}_{11}(z)&l^{\pm}_{12}(z)&\cdots&l^{\pm}_{18}(z)\\l^{\pm}_{21}(z)&l^{\pm}_{22}(z)&\ddots&\vdots \\ \vdots&\ddots&\ddots&l^{\pm}_{78}(z)\\ l^{\pm}_{81}(z)&\cdots&l^{\pm}_{87}(z)&l^{\pm}_{88}(z) \end{array}\right)_{8\times8}.
\end{align*}
Observe (5.9) and (5.11) and restrict them to $E_{ij}\otimes E_{kl}$, $1\le i,j,k,l\le 4$, then
\begin{align*}
\hat{R}_1(\frac{z}{w}){\widetilde{L}}^{\pm}_1(z){\widetilde{L}}^{\pm}_2(w)
&={\widetilde{L}}^{\pm}_2(w){\widetilde{L}}^{\pm}_1(z)\hat{R}_1(\frac{z}{w}),\\
\hat{R}_1(\frac{z_+}{w_-}){\widetilde{L}}^{\pm}_1(z){\widetilde{L}}^{\pm}_2(w)
&={\widetilde{L}}^{\pm}_2(w){\widetilde{L}}^{\pm}_1(z)\hat{R}_1(\frac{z_-}{w_+}),
\end{align*}
\begin{align*}
\widetilde{L}^{\pm}(z)=\left(\begin{array}{cccc}l^{\pm}_{11}(z)&l^{\pm}_{12}(z)&l^{\pm}_{13}(z)&l^{\pm}_{14}(z)\\l^{\pm}_{21}(z)&l^{\pm}_{22}(z)&l^{\pm}_{23}(z)&l^{\pm}_{24}(z) \\ l^{\pm}_{31}(z)&l^{\pm}_{32}(z)&l^{\pm}_{33}(z)&l^{\pm}_{34}(z)\\ l^{\pm}_{41}(z)&l^{\pm}_{42}(z)&l^{\pm}_{43}(z)&l^{\pm}_{44}(z) \end{array}\right)_{4\times4},
\end{align*}
and
\begin{align*}
\hat{R}_1(\frac{z}{w})=&\sum\limits_{i=1}^{4}E_{ii}\otimes E_{ii}+\frac{s(w-z)}{sw-zr}\Bigl(\sum\limits_{i>j}E_{ii}\otimes E_{jj}+s^{-1}\sum\limits_{i<j}E_{ii}\otimes E_{jj}\Bigr)\\
&+\frac{(s-r)w}{sw-zr}\Bigl(\sum\limits_{i>j}E_{ij}\otimes E_{ji}+\frac{z}{w}\sum\limits_{i<j}E_{ij}\otimes E_{ji}\Bigr).
\end{align*}
Jing and Liu \cite{JL} gave the following spectral parameter dependent $\hat{R}_{A}(\frac{z}{w})$ of $U_{r,s}(\widehat{\mathfrak{gl}}_n)$, in particular, setting $n=4$,
\begin{align*}
\hat{R}_{A}(\frac{z}{w})=&\sum\limits_{i=1}^{4}E_{ii}\otimes E_{ii}+\frac{w-z}{w-zrs^{-1}}\Bigl(\sum\limits_{i>j}E_{ii}\otimes E_{jj}+s^{-1}\sum\limits_{i<j}E_{ii}\otimes E_{jj}\Bigr)\\
&+\frac{(1-rs^{-1})w}{w-zrs^{-1}}\Bigl(\sum\limits_{i>j}E_{ij}\otimes E_{ji}+\frac{z}{w}\sum\limits_{i<j}E_{ij}\otimes E_{ji}\Bigr),
\end{align*}
we get $\hat{R}_{A}(\frac{z}{w})$=$\hat{R}_1(\frac{z}{w})$. Thereby, we can directly have the following relations of generators $\Bigl\{X^{\pm}_1(z), X^{\pm}_2(z), X^{\pm}_3(z), k^{\pm}_{1}(z), k^{\pm}_{2}(z), k^{\pm}_{3}(z), k^{\pm}_{4}(z)\Bigr\}$.

Next, we need to obtain the relations between the remaining Gauss generators. First, using the Gauss decomposition, we write down $L^{\pm}(z)$ and $L^{\pm}(z)^{-1}$:
\begin{align*}
L^{\pm}(z)=&\left(\begin{array}{ccc}k^{\pm}_1(z)&k^{\pm}_1(z)e^{\pm}_{12}(z)& \cdots \\f^{\pm}_{21}(z)k^{\pm}_1(z)&\vdots&\vdots\\ \vdots &\vdots&\vdots \end{array}\right),
\end{align*}
and
\begin{align*}
L^{\pm}(z)^{-1}=&\left(\begin{array}{ccc}\cdots&\cdots& \cdots \\ \cdots&\cdots&-e^{\pm}_{N-1,N}(z)k^{\pm}_N(z)^{-1}\\ \cdots &-k^{\pm}_N(z)^{-1}f^{\pm}_{N,N-1}(z)& k^{\pm}_N(z)^{-1}\end{array}\right).
\end{align*}
Then using the generating series relations (5.2) and (5.4), we can complete our proof by using the following lemmas.
\end{proof}
\begin{lemma}
The following equations hold in $\mathcal{U}(\hat R)$:
\begin{align}
k^{\pm}_5(w)X^{\pm}_i(z)&=X^{\pm}_i(z)k^{\pm}_5(w),\\
k^{\pm}_5(w)X^{\mp}_i(z)&=X^{\mp}_i(z)k^{\pm}_5(w),\\
k^{\pm}_j(z)k^{\pm}_5(w)&=k^{\pm}_5(w)k^{\pm}_j(z),\\
k^{\pm}_j(z)k^{\mp}_5(w)\frac{z_{\pm}-w_{\mp}}{rz_{\pm}-sw_{\mp}}&=\frac{z_{\mp}-w_{\pm}}{rz_{\mp}-sw_{\pm}}k^{\mp}_5(w)k^{\pm}_j(z),
\end{align}
where $1\le i\le2$ and $1\le j\le 3$.
\end{lemma}
\begin{proof}
Due to the observation made in formulas (5.9) and (5.11), the relations between the Gaussian generators mentioned above follow from those for the quantum affine algebra $U_{r,s}(\widehat{\mathfrak{gl}_{5}})$; see \cite{JL}.
\end{proof}
\begin{lemma}
\begin{align}
k^{\pm}_1(w)X^{+}_4(z)&=rsX^{+}_4(z)k^{\pm}_1(w),\\
rsk^{\pm}_1(w)X^{-}_4(z)&=X^{-}_4(z)k^{\pm}_1(w),\\
k^{\pm}_2(w)X^{+}_4(z)&=rsX^{+}_4(z)k^{\pm}_2(w),\\
rsk^{\pm}_2(w)X^{-}_4(z)&=X^{-}_4(z)k^{\pm}_2(w),\\
X^{\pm}_4(w)X^{\pm}_1(z)&=X^{\pm}_1(z)X^{\pm}_4(w),\\
X^{\pm}_4(w)X^{\mp}_1(z)&=X^{\mp}_1(z)X^{\pm}_4(w).
\end{align}
\end{lemma}
\begin{proof}
We only give details for one case of (6.9), (6.11), (6.13) and (6.14), the remaining relations are verified in a similar way. By (5.9) and (5.11), taking the equations $M_{11}=M^{\prime}_{11}$, we get
\begin{align*}
a_{15}(\frac{z}{w})k^{\pm}_3(w)e^{\pm}_{3,5}(w)l^{\pm}_{11}(z)=a_{13}(\frac{z}{w})l^{\pm}_{11}(z)k^{\pm}_3(w)e^{\pm}_{3,5}(w),
\end{align*}
and the relations between $k^{\pm}_3(w)e^{\pm}_{3,5}(w)$ and $l^{\mp}_{11}(z)$ we have
\begin{align*}
a_{15}(\frac{z_{\mp}}{w_{\pm}})k^{\pm}_3(w)e^{\pm}_{3,5}(w)l^{\mp}_{11}(z)=a_{13}(\frac{z_{\pm}}{w_{\mp}})l^{\mp}_{11}(z)k^{\pm}_3(w)e^{\pm}_{3,5}(w).
\end{align*}
so that $k^{+}_1(z)X^{+}_4(w)=rsX^{+}_4(w)k^{+}_1(z)$. Now apply $M_{12}=M^{\prime}_{12}$ to obtain
\begin{align*}
a_{25}(\frac{z}{w})k^{\pm}_3(w)e^{\pm}_{3,5}(w)l^{\pm}_{12}(z)=a_{13}(\frac{z}{w})l^{\pm}_{12}(z)k^{\pm}_3(w)e^{\pm}_{3,5}(w),
\end{align*}
and
\begin{align*}
a_{25}(\frac{z_{\mp}}{w_{\pm}})k^{\pm}_3(w)e^{\pm}_{3,5}(w)l^{\mp}_{12}(z)=a_{13}(\frac{z_{\pm}}{w_{\mp}})l^{\mp}_{12}(z)k^{\pm}_3(w)e^{\pm}_{3,5}(w).
\end{align*}
So we get $X^{+}_4(w)X^{+}_1(z)=X^{+}_1(z)X^{+}_4(w)$. Since $M_{21}=M^{\prime}_{21}$, we get
\begin{align*}
a_{15}(\frac{z}{w})k^{\pm}_3(w)e^{\pm}_{3,5}(w)l^{\pm}_{21}(z)=a_{23}(\frac{z}{w})l^{\pm}_{21}(z)k^{\pm}_3(w)e^{\pm}_{3,5}(w).
\end{align*}
Furthermore,
\begin{align*}
a_{15}(\frac{z_{\mp}}{w_{\pm}})k^{\pm}_3(w)e^{\pm}_{3,5}(w)l^{\mp}_{21}(z)=a_{23}(\frac{z_{\pm}}{w_{\mp}})l^{\mp}_{21}(z)k^{\pm}_3(w)e^{\pm}_{3,5}(w),
\end{align*}
thus we prove a case of relation in (6.14). Since $M_{22}=M^{\prime}_{22}$, the formula is
\begin{align*}
a_{25}(\frac{z}{w})k^{\pm}_3(w)e^{\pm}_{3,5}(w)l^{\pm}_{22}(z)=a_{23}(\frac{z}{w})l^{\pm}_{22}(z)k^{\pm}_3(w)e^{\pm}_{3,5}(w),
\end{align*}
on the other hand, we have
\begin{align*}
a_{25}(\frac{z_{\mp}}{w_{\pm}})k^{\pm}_3(w)e^{\pm}_{3,5}(w)l^{\mp}_{22}(z)=a_{23}(\frac{z_{\pm}}{w_{\mp}})l^{\mp}_{22}(z)k^{\pm}_3(w)e^{\pm}_{3,5}(w).
\end{align*}
As a final step, use the relations between $l^{\pm}_{21}(z)$ and $X^{+}_4(w)$ and those between $e^{\mp}_{12}(z)$ and $X^{+}_4(w)$, to come to the relation
\begin{align*}
k^{+}_2(z)X^{+}_4(w)&=rsX^{+}_4(w)k^{+}_2(z),
\end{align*}
as required.
\end{proof}
\begin{lemma}
\begin{align}
(rw-sz_{\pm})k^{\pm}_3(w)X^{+}_4(z)&=rs(w-z_{\pm})X^{+}_4(z)k^{\pm}_3(w),\\
rs(w-z_{\mp})k^{\pm}_3(w)X^{-}_4(z)&=(rw-sz_{\mp})X^{-}_4(z)k^{\pm}_3(w),\\
X^{\pm}_4(w)X^{\mp}_2(z)&=X^{\mp}_2(z)X^{\pm}_4(w),\\
(z-w)X^{+}_4(w)X^{+}_2(z)&=X^{+}_2(z)X^{+}_4(w)(rz-sw)\nonumber,\\
(rz-sw)X^{-}_4(w)X^{-}_2(z)&=X^{-}_2(z)X^{-}_4(w)(z-w).
\end{align}
\end{lemma}
\begin{proof}
The arguments for both formulas are quite similar so we only give a proof of one case of (6.17) and (6.18), taking the equations $M_{13}=M^{\prime}_{13}$, we get
\begin{align}
a_{35}(\frac{z}{w})l^{\pm}_{15}(w)l^{\pm}_{13}(z)+b_{53}(\frac{z}{w})l^{\pm}_{13}(w)l^{\pm}_{15}(z)&=l^{\pm}_{13}(z)l^{\pm}_{15}(w),\\
a_{35}(\frac{z}{w})l^{\pm}_{25}(w)l^{\pm}_{13}(z)+b_{53}(\frac{z}{w})l^{\pm}_{23}(w)l^{\pm}_{15}(z)&=a_{12}(\frac{z}{w})l^{\pm}_{13}(z)l^{\pm}_{25}(w)+b_{12}(\frac{z}{w})l^{\pm}_{23}(z)l^{\pm}_{15}(w),\\
a_{35}(\frac{z}{w})l^{\pm}_{35}(w)l^{\pm}_{13}(z)+b_{53}(\frac{z}{w})l^{\pm}_{33}(w)l^{\pm}_{15}(z)&=a_{13}(\frac{z}{w})l^{\pm}_{13}(z)l^{\pm}_{35}(w)+b_{13}(\frac{z}{w})l^{\pm}_{33}(z)l^{\pm}_{15}(w).
\end{align}
Using $f^{\pm}_{32}(w)f^{\pm}_{21}(w)\cdot(6.19)-f^{\pm}_{32}(w)(6.20)-f^{\pm}_{31}(w)\cdot(6.19)+(6.21)$, through a lot of calculations, we can obtain
\begin{align}
&a_{35}(\frac{z}{w})k^{\pm}_3(w)e^{\pm}_{35}(w)l^{\pm}_{13}(z)+b_{53}(\frac{z}{w})k^{\pm}_{3}(w)l^{\pm}_{15}(z)+a_{13}(\frac{z}{w})f^{\pm}_{32}(w)l^{\pm}_{13}(z)k^{\pm}_2(w)e^{\pm}_{25}(w)\nonumber\\
&\quad=a_{13}(\frac{z}{w})l^{\pm}_{13}(z)f^{\pm}_{32}(w)k^{\pm}_2(w)e^{\pm}_{25}(w)+a_{13}(\frac{z}{w})l^{\pm}_{13}(z)k^{\pm}_3(w)e^{\pm}_{35}(w)+b_{13}(\frac{z}{w})k^{\pm}_{3}(w)l^{\pm}_{11}(z)e^{\pm}_{15}(w).
\end{align}
Taking the equations $M_{23}=M^{\prime}_{23}$, we have
\begin{align}
a_{35}(\frac{z}{w})l^{\pm}_{15}(w)l^{\pm}_{23}(z)+b_{53}(\frac{z}{w})l^{\pm}_{13}(w)l^{\pm}_{25}(z)&=a_{21}(\frac{z}{w})l^{\pm}_{23}(z)l^{\pm}_{15}(w)+b_{21}(\frac{z}{w})l^{\pm}_{13}(z)l^{\pm}_{25}(w),\\
a_{35}(\frac{z}{w})l^{\pm}_{25}(w)l^{\pm}_{23}(z)+b_{53}(\frac{z}{w})l^{\pm}_{23}(w)l^{\pm}_{25}(z)&=l^{\pm}_{23}(z)l^{\pm}_{25}(w),\\
a_{35}(\frac{z}{w})l^{\pm}_{35}(w)l^{\pm}_{23}(z)+b_{53}(\frac{z}{w})l^{\pm}_{33}(w)l^{\pm}_{25}(z)&=a_{23}(\frac{z}{w})l^{\pm}_{23}(z)l^{\pm}_{35}(w)+b_{23}(\frac{z}{w})l^{\pm}_{33}(z)l^{\pm}_{25}(w).
\end{align}
A similar calculation, we come to the relation
\begin{align}
a_{35}(\frac{z}{w})k^{\pm}_3(w)e^{\pm}_{35}(w)l^{\pm}_{23}(z){+}b_{53}(\frac{z}{w})k^{\pm}_{3}(w)l^{\pm}_{25}(z)=b_{21}(\frac{z}{w})\Bigl\{f^{\pm}_{32}(w)f^{\pm}_{21}(w)l^{\pm}_{13}(z)k^{\pm}_2(w)e^{\pm}_{25}(w)&\nonumber\\
-f^{\pm}_{31}(w)l^{\pm}_{13}(z)k^{\pm}_2(w)e^{\pm}_{25}(w)\Bigr\}
{+}a_{23}(\frac{z}{w})\Bigl\{l^{\pm}_{23}(z)f^{\pm}_{32}(w)k^{\pm}_2(w)e^{\pm}_{25}(w)
{+}l^{\pm}_{23}(z)k^{\pm}_{3}(w)e^{\pm}_{35}(w)\Bigr\}&\nonumber\\
+b_{23}(\frac{z}{w})\Bigl\{l^{\pm}_{33}(z)k^{\pm}_2(w)e^{\pm}_{25}(w){+}l^{\pm}_{21}(z)k^{\pm}_3(w)e^{\pm}_{15}(w)\Bigr\}
{-}f^{\pm}_{32}(w)l^{\pm}_{23}(z)k^{\pm}_2(w)e^{\pm}_{25}(w).&
\end{align}
Furthermore, by using $-f^{\pm}_{21}(z)\cdot(6.22)+(6.26)$, we get
\begin{align*}
a_{35}&(\frac{z}{w})k^{\pm}_3(w)e^{\pm}_{35}(w)k^{\pm}_2(z)e^{\pm}_{23}(z)+b_{53}(\frac{z}{w})k^{\pm}_{3}(w)k^{\pm}_2(z)e^{\pm}_{25}(z)\nonumber\\
&=b_{23}(\frac{z}{w})k^{\pm}_2(z)k^{\pm}_3(w)e^{\pm}_{25}(w)+a_{23}(\frac{z}{w})k^{\pm}_2(z)
e^{\pm}_{23}(z)k^{\pm}_3(w)e^{\pm}_{35}(w),
\end{align*}
and taking into account the relations between $e^{\mp}_{3,5}(w)$ and $e^{\pm}_{23}(z)$, we have
\begin{align*}
a_{35}(\frac{z_{\mp}}{w_{\pm}})&k^{\mp}_3(w)e^{\mp}_{35}(w)k^{\pm}_2(z)e^{\pm}_{23}(z)+b_{53}(\frac{z_{\mp}}{w_{\pm}})k^{\mp}_{3}(w)k^{\pm}_2(z)e^{\pm}_{25}(z)\nonumber\\
&=b_{23}(\frac{z_{\pm}}{w_{\mp}})k^{\pm}_2(z)k^{\mp}_3(w)e^{\mp}_{25}(w)+a_{23}(\frac{z_{\pm}}{w_{\mp}})k^{\pm}_2(z)e^{\pm}_{23}(z)k^{\mp}_3(w)e^{\mp}_{35}(w).
\end{align*}
Therefore, we can arrive at $(z-w)X^{+}_4(w)X^{+}_2(z)=X^{+}_2(z)X^{+}_4(w)(rz-sw)$. Now turn to (6.24).  Taking the equation $M_{31}=M^{\prime}_{31}$, we give
\begin{align}
a_{15}(\frac{z_{\mp}}{w_{\pm}})k^{\mp}_3(w)e^{\mp}_{35}(w)l^{\pm}_{31}(z)&=\Bigl\{l^{\pm}_{31}(z)-b_{32}(\frac{z_{\pm}}{w_{\mp}})f^{\mp}_{32}(w)l^{\pm}_{21}(z)\nonumber\\
                                                     &\ +b_{31}(\frac{z_{\pm}}{w_{\mp}})\Bigl(f^{\mp}_{32}(w)f^{\mp}_{21}(w)-f^{\mp}_{31}(w)\Bigr)l^{\pm}_{11}(z)\Bigr\}k^{\mp}_3(w)e^{\mp}_{35}(w),\\
a_{15}(\frac{z}{w})k^{\pm}_3(w)e^{\pm}_{35}(w)l^{\pm}_{31}(z)&=\Bigl\{l^{\pm}_{31}(z)-b_{32}(\frac{z}{w})f^{\pm}_{32}(w)l^{\pm}_{21}(z)\nonumber\\
                                                       &\ +b_{31}(\frac{z}{w})\Bigl(f^{\pm}_{32}(w)f^{\pm}_{21}(w)-f^{\pm}_{31}(w)\Bigr)l^{\pm}_{11}(z)\Bigr\}k^{\pm}_3(w)e^{\pm}_{35}(w).
\end{align}
Using $M_{32}=M^{\prime}_{32}$, consider the relations
\begin{align}
a_{15}(\frac{z_{\mp}}{w_{\pm}})k^{\mp}_3(w)e^{\mp}_{35}(w)l^{\pm}_{32}(z)&=\Bigl\{l^{\pm}_{32}(z)-b_{32}(\frac{z_{\pm}}{w_{\mp}})f^{\mp}_{32}(w)l^{\pm}_{22}(z)\nonumber\\
                                                     &\ +b_{31}(\frac{z_{\pm}}{w_{\mp}})\Bigl(f^{\mp}_{32}(w)f^{\mp}_{21}(w)-f^{\mp}_{31}(w)\Bigr)l^{\pm}_{12}(z)\Bigr\}k^{\mp}_3(w)e^{\mp}_{35}(w),\\
a_{15}(\frac{z}{w})k^{\pm}_3(w)e^{\pm}_{35}(w)l^{\pm}_{32}(z)&=\Bigl\{l^{\pm}_{32}(z)-b_{32}(\frac{z}{w})f^{\pm}_{32}(w)l^{\pm}_{22}(z)\nonumber\\
                                                       &\ +b_{31}(\frac{z}{w})\Bigl(f^{\pm}_{32}(w)f^{\pm}_{21}(w)
                                                       -f^{\pm}_{31}(w)\Bigr)l^{\pm}_{12}(z)\Bigr\}k^{\pm}_3(w)e^{\pm}_{35}(w).
\end{align}
By $-(6.27)\cdot e^{\pm}_{12}(z)+(6.29)$ and $-(6.28)\cdot e^{\pm}_{12}(z)+(6.30)$, we can obtain separately,
\begin{align*}
a_{15}(\frac{z}{w})k^{\pm}_3(w)e^{\pm}_{35}(w)f^{\pm}_{32}(z)k^{\pm}_2(z)&=\Bigl\{f^{\pm}_{32}(z)k^{\pm}_2(z){-}b_{32}(\frac{z}{w})f^{\pm}_{32}(w)k^{\pm}_{2}(z)\Bigr\}k^{\pm}_3(w)e^{\pm}_{35}(w),\\
a_{15}(\frac{z_{\mp}}{w_{\pm}})k^{\mp}_3(w)e^{\mp}_{35}(w)f^{\pm}_{32}(z)k^{\pm}_2(z)&=\Bigl\{f^{\pm}_{32}(z)
k^{\pm}_2(z){-}b_{32}(\frac{z_{\pm}}{w_{\mp}})f^{\mp}_{32}(w)k^{\pm}_{2}(z)\Bigr\}k^{\mp}_3(w)e^{\mp}_{35}(w),
\end{align*}
and therefore $X^{+}_4(w)X^{-}_2(z)=X^{-}_2(z)X^{+}_4(w)$.
\end{proof}
\begin{lemma}
In the algebra $\mathcal{U}(\hat R)$, we have
\begin{align*}
e^{\pm}_{36}(z)=f^{\pm}_{63}(z)=0,\quad e^{\pm}_{45}(z)=f^{\pm}_{54}(z)=0.
\end{align*}
\begin{proof}
We only verify a case of the first relation. By (5.9) and (5.11), we have $M_{13}=M^{\prime}_{13}$, and so get the relations:
\begin{align}
\sum\limits_{i=1}^{8}c_{i6}(\frac{z}{w})l^{\pm}_{1i}(w)l^{\pm}_{1i^{\prime}}(z)&=l^{\pm}_{13}(z)l^{\pm}_{16}(w),\\
\sum\limits_{i=1}^{8}c_{i6}(\frac{z}{w})l^{\pm}_{2i}(w)l^{\pm}_{1i^{\prime}}(z)&=a_{12}(\frac{z}{w})l^{\pm}_{13}(z)l^{\pm}_{26}(w)+b_{12}(\frac{z}{w})l^{\pm}_{23}(z)l^{\pm}_{16}(w),\\
\sum\limits_{i=1}^{8}c_{i6}(\frac{z}{w})l^{\pm}_{3i}(w)l^{\pm}_{1i^{\prime}}(z)&=a_{13}(\frac{z}{w})l^{\pm}_{13}(z)l^{\pm}_{36}(w)+b_{13}(\frac{z}{w})l^{\pm}_{33}(z)l^{\pm}_{16}(w),
\end{align}
let $\Bigl(f^{\pm}_{32}(w)f^{\pm}_{21}(w)-f^{\pm}_{31}(w)\Bigr)\cdot (6.31)-f^{\pm}_{32}(w)\cdot (6.32)+(6.33)$. Through a lot of calculations, we obtain
\begin{align}
\sum\limits_{i=3}^{8}c_{i6}(\frac{z}{w})k^{\pm}_{3}(w)e^{\pm}_{3i}(w)l^{\pm}_{1i^{\prime}}(z)=a_{13}&(\frac{z}{w})l^{\pm}_{13}(z)k^{\pm}_3(w)e^{\pm}_{36}(w)+b_{13}(\frac{z}{w})k^{\pm}_{3}(w)k^{\pm}_1(z)\nonumber\\                                                      &\cdot\Bigl\{\Bigl(e^{\pm}_{12}(w)-e^{\pm}_{12}(z)\Bigr)e^{\pm}_{26}(w)-e^{\pm}_{16}(w)\Bigr\}.
\end{align}
And from $M_{23}=M^{\prime}_{23}$, we obtain
\begin{align}
\sum\limits_{i=1}^{8}c_{i6}(\frac{z}{w})l^{\pm}_{1i}(w)l^{\pm}_{2i^{\prime}}(z)&=a_{21}(\frac{z}{w})l^{\pm}_{23}(z)l^{\pm}_{16}(w)+b_{21}(\frac{z}{w})l^{\pm}_{13}(z)l^{\pm}_{26}(w),\\
\sum\limits_{i=1}^{8}c_{i6}(\frac{z}{w})l^{\pm}_{2i}(w)l^{\pm}_{2i^{\prime}}(z)&=l^{\pm}_{23}(z)l^{\pm}_{26}(w),\\
\sum\limits_{i=1}^{8}c_{i6}(\frac{z}{w})l^{\pm}_{3i}(w)l^{\pm}_{2i^{\prime}}(z)&=a_{23}(\frac{z}{w})l^{\pm}_{23}(z)l^{\pm}_{36}(w)+b_{23}(\frac{z}{w})l^{\pm}_{33}(z)l^{\pm}_{26}(w),
\end{align}
Calculating $\Bigl(f^{\pm}_{32}(w)f^{\pm}_{21}(w)-f^{\pm}_{31}(w)\Bigr)\cdot (6.35)-f^{\pm}_{32}(w)\cdot (6.36)+(6.37)-f^{\pm}_{21}(z)\cdot (6.34)$, we get
\begin{align}
\sum\limits_{i=3}^{8}c_{i6}(\frac{z}{w})k^{\pm}_{3}(w)e^{\pm}_{3i}(w)k^{\pm}_2(z)e^{\pm}_{2i^{\prime}}(z)
=a_{23}&(\frac{z}{w})k^{\pm}_{2}(z)e^{\pm}_{23}(z)k^{\pm}_{3}(w)e^{\pm}_{36}(w)\nonumber\\
&+b_{23}(\frac{z}{w})k^{\pm}_{3}(w)k^{\pm}_{2}(z)e^{\pm}_{26}(w).
\end{align}
Using $M_{23}=M^{\prime}_{23}$, we come by the relations
\begin{align}
\sum\limits_{i=1}^{8}c_{i6}(\frac{z}{w})l^{\pm}_{1i}(w)l^{\pm}_{3i^{\prime}}(z)&
=a_{31}(\frac{z}{w})l^{\pm}_{33}(z)l^{\pm}_{16}(w)+b_{31}(\frac{z}{w})l^{\pm}_{13}(z)l^{\pm}_{36}(w),\\
\sum\limits_{i=1}^{8}c_{i6}(\frac{z}{w})l^{\pm}_{2i}(w)l^{\pm}_{3i^{\prime}}(z)&
=a_{32}(\frac{z}{w})l^{\pm}_{33}(z)l^{\pm}_{26}(w)+b_{32}(\frac{z}{w})l^{\pm}_{23}(z)l^{\pm}_{36}(w),\\
\sum\limits_{i=1}^{8}c_{i6}(\frac{z}{w})l^{\pm}_{3i}(w)l^{\pm}_{3i^{\prime}}(z)&=l^{\pm}_{33}(z)l^{\pm}_{36}(w).
\end{align}
Calculating $\Bigl(f^{\pm}_{32}(w)f^{\pm}_{21}(w)-f^{\pm}_{31}(w)\Bigr)\cdot (6.39)-f^{\pm}_{32}(w)\cdot (6.40)+(6.41)-f^{\pm}_{32}(z)\cdot (6.38)-f^{\pm}_{31}(z)\cdot(6.34)$, we arrive at the relation:
\begin{align}
\sum\limits_{i=3}^{8}c_{i6}(\frac{z}{w})k^{\pm}_{3}(w)e^{\pm}_{3i}(w)l^{\pm}_{3i^{\prime}}(z)=k^{\pm}_{3}(z)k^{\pm}_{3}(w)e^{\pm}_{36}(w).
\end{align}
Setting $z=w$, we get $e^{\pm}_{36}(z)=0$. The remaining relations can be proved in a similar way.
\end{proof}
\end{lemma}
\begin{lemma}
The following equations hold in $\mathcal{U}(\hat R)$:
\begin{align}
X^{\pm}_{4}(w)X^{\pm}_3(z)&=(rs)^{\pm 1}X^{\pm}_3(z)X^{\pm}_4(w),\\
X^{\pm}_3(w)X^{\mp}_4(z)&=X^{\mp}_{4}(z)X^{\pm}_3(w).
\end{align}
\end{lemma}
\begin{proof}
The arguments are similar for all relations so we only prove the relation $X^{+}_{4}(w)X^{+}_3(z)=rsX^{+}_3(z)X^{+}_4(w)$. By Eqs. (5.9)---(5.12), we have $M_{34}=M^{\prime}_{34}$, and then get the relation:
\begin{align*}
\sum\limits_{i=1}^{8}c_{i5}(\frac{z}{w})l^{\pm}_{3i}(w)l^{\pm}_{3i^{\prime}}(z)=l^{\pm}_{34}(z)l^{\pm}_{35}(w),
\end{align*}
through a similar calculation process in Lemma 6.5, which yields
\begin{align}
\sum\limits_{i=3}^{8}c_{i5}(\frac{z}{w})k^{\pm}_{3}(w)e^{\pm}_{3i}(w)k^{\pm}_{3}(z)e^{\pm}_{3i^{\prime}}(z)=k^{\pm}_{3}(z)e^{\pm}_{34}(z)k^{\pm}_{3}(w)e^{\pm}_{35}(w),
\end{align}
and from $M_{35}=M^{\prime}_{35}$, we obtain
\begin{align}
\sum\limits_{i=3}^{8}c_{i4}(\frac{z}{w})k^{\pm}_{3}(w)e^{\pm}_{3i}(w)k^{\pm}_{3}(z)e^{\pm}_{3i^{\prime}}(z)=k^{\pm}_{3}(z)e^{\pm}_{35}(z)k^{\pm}_{3}(w)e^{\pm}_{34}(w).
\end{align}
Furthermore, $M_{33}=M^{\prime}_{33}$ and $M_{36}=M^{\prime}_{36}$ give that
\begin{align}
\sum\limits_{i=3}^{8}c_{i6}(\frac{z}{w})k^{\pm}_{3}(w)e^{\pm}_{3i}(w)k^{\pm}_{3}(z)e^{\pm}_{3i^{\prime}}(z)=k^{\pm}_{3}(z)k^{\pm}_{3}(w)e^{\pm}_{36}(w),
\end{align}
and
\begin{align}
\sum\limits_{i=3}^{8}c_{i3}(\frac{z}{w})k^{\pm}_{3}(w)e^{\pm}_{3i}(w)k^{\pm}_{3}(z)e^{\pm}_{3i^{\prime}}(z)=k^{\pm}_{3}(z)e^{\pm}_{36}(z)k^{\pm}_{3}(w).
\end{align}
Combining (6.45) with (6.46), we get that
\begin{align}
(rs)^{-1}X^{+}_{4}(w)X^{+}_{3}(z)-X^{+}_{3}(w)X^{+}_{4}(z)=X^{+}_{3}(z)X^{+}_{4}(w)-(rs)^{-1}X^{+}_{4}(z)X^{+}_{3}(w).
\end{align}
Taking $(6.45)$ and $(6.47)$, owing to Lemma 6.5 and the fact that $e^{\pm}_{36}(z)=0$, we have
\begin{align}
\sum\limits_{i=4}^{5}k^{\pm}_{3}(w)&e^{\pm}_{3i}(w)k^{\pm}_{3}(z)e^{\pm}_{3i^{\prime}}(z)\Bigl(c_{i5}(\frac{z}{w})-(r^{-1}s)^{\frac{1}{2}}c_{i6}(\frac{z}{w})\Bigr)=k^{\pm}_{3}(z)e^{\pm}_{34}(z)k^{\pm}_{3}(w)e^{\pm}_{35}(w).
\end{align}
Using the relations $(6.46)$ and $(6.48)$, we get
\begin{align}
\sum\limits_{i=4}^{5}k^{\pm}_{3}(w)&e^{\pm}_{3i}(w)k^{\pm}_{3}(z)e^{\pm}_{3i^{\prime}}(z)\Bigl(c_{i4}(\frac{z}{w})-(rs^{-1})^{\frac{1}{2}}c_{i3}(\frac{z}{w})\Bigr)=k^{\pm}_{3}(z)e^{\pm}_{35}(z)k^{\pm}_{3}(w)e^{\pm}_{34}(w).
\end{align}
Exchanging $z$ and $w$ in the relation (6.51), and combining into (6.50), we obtain
\begin{align}
z(rs)^{-1}X^{+}_{4}(w)X^{+}_{3}(z)-wX^{+}_{3}(w)X^{+}_{4}(z)=zX^{+}_{3}(z)X^{+}_{4}(w)-w(rs)^{-1}X^{+}_{4}(z)X^{+}_{3}(w)
\end{align}
By (6.49) and (6.52), we get $X^{+}_{4}(w)X^{+}_{3}(z)=rsX^{+}_{3}(z)X^{+}_{4}(w)$.
\end{proof}
\begin{lemma}
The following equations hold in $\mathcal{U}(\hat R)$:
\begin{align}
[X^{+}_{4}(z),X^{-}_{4}(w)]=(s^{-1}{-}r^{-1})\Bigl\{\delta\Bigl(\frac{z_{-}}{w_+}\Bigr)k^{-}_{5}(w_+)k^{-}_3(w_+)^{-1}
{-}\delta\Bigl(\frac{z_{+}}{w_-}\Bigr)k^{+}_{5}(z_+)k^{+}_3(z_+)^{-1}\Bigr\}.
\end{align}
\end{lemma}
\begin{proof}
By (5.10) and (5.12), we have $M_{35}=M^{\prime}_{35}$ and then get the relations:
\begin{align*}
a_{53}(\frac{z}{w})l^{\pm}_{13}(w)l^{\pm}_{35}(z){+}b_{35}(\frac{z}{w})l^{\pm}_{15}(w)l^{\pm}_{33}(z)&
=b_{31}(\frac{z}{w})l^{\pm}_{15}(z)l^{\pm}_{33}(w){+}a_{31}(\frac{z}{w})l^{\pm}_{35}(z)l^{\pm}_{13}(w),\\
a_{53}(\frac{z}{w})l^{\pm}_{23}(w)l^{\pm}_{35}(z){+}b_{35}(\frac{z}{w})l^{\pm}_{25}(w)l^{\pm}_{33}(z)&
=b_{32}(\frac{z}{w})l^{\pm}_{25}(z)l^{\pm}_{33}(w){+}a_{32}(\frac{z}{w})l^{\pm}_{35}(z)l^{\pm}_{23}(w),\\
a_{53}(\frac{z}{w})l^{\pm}_{33}(w)l^{\pm}_{35}(z){+}b_{35}(\frac{z}{w})l^{\pm}_{35}(w)l^{\pm}_{33}(z)&
=l^{\pm}_{35}(z)l^{\pm}_{33}(w),\\
a_{53}(\frac{z}{w})l^{\pm}_{53}(w)l^{\pm}_{35}(z){+}b_{35}(\frac{z}{w})l^{\pm}_{55}(w)l^{\pm}_{33}(z)&
=b_{35}(\frac{z}{w})l^{\pm}_{55}(z)l^{\pm}_{33}(w){+}a_{35}(\frac{z}{w})l^{\pm}_{35}(z)l^{\pm}_{53}(w).
\end{align*}
By straightforward calculations, one can check that
\begin{align*}
[X^{+}_{4}(z),X^{-}_{4}(w)]=(s^{-1}{-}r^{-1})\Bigl\{\delta\Bigl(\frac{z_{-}}{w_+}\Bigr)k^{-}_{5}(w_+)k^{-}_3(w_+)^{-1}
{-}\delta\Bigl(\frac{z_{+}}{w_-}\Bigr)k^{+}_{5}(z_+)k^{+}_3(z_+)^{-1}\Bigr\}.
\end{align*}

This completes the proof.
\end{proof}
\begin{lemma}
The following equations hold in $\mathcal{U}(\hat R)$:
\begin{align}
k^{\pm}_5(w)X^{+}_3(z)&=\frac{rs(z_{\pm}-w)}{z_{\pm}s-rw}X^{+}_3(z)k^{\pm}_5(w),\\
X^{-}_3(z)k^{\pm}_5(w)&=\frac{rs(z_{\mp}-w)}{z_{\mp}s-rw}k^{\pm}_5(w)X^{-}_3(z),\\
k^{\pm}_4(w)X^{+}_4(z)&=\frac{wr-sz_{\pm}}{w-z_{\pm}}X^{+}_4(z)k^{\pm}_4(w),\\
X^{-}_4(z)k^{\pm}_4(w)&=\frac{wr-sz_{\mp}}{w-z_{\mp}}k^{\pm}_4(w)X^{-}_4(z).
\end{align}
\end{lemma}
\begin{proof}
We only give a proof of one case of (6.54). Similarly, we give the other identities. Using again $M_{35}=M^{\prime}_{35}$ and $M_{34}=M^{\prime}_{34}$, we get
\begin{align*}
\sum\limits_{i=1}^{8}c_{i5}(\frac{z}{w})l^{\pm}_{5i}(w)l^{\pm}_{3i^{\prime}}(z)=b_{35}(\frac{z}{w})l^{\pm}_{54}(z)l^{\pm}_{35}(w)+a_{35}(\frac{z}{w})l^{\pm}_{34}(z)l^{\pm}_{55}(w)
\end{align*}
through the same calculating process in Lemma 6.5, which yields
\begin{align}
\sum\limits_{i=5}^{8}c_{i5}(\frac{z}{w})k^{\pm}_{5}(w)e^{\pm}_{5i}(w)k^{\pm}_{3}(z)e^{\pm}_{3i^{\prime}}(z)=a_{35}(\frac{z}{w})k^{\pm}_{3}(z)e^{\pm}_{34}(z)k^{\pm}_{5}(w).
\end{align}
From $M_{35}=M^{\prime}_{35}$, we obtain
\begin{align*}
\sum\limits_{i=1}^{8}c_{i4}(\frac{z}{w})l^{\pm}_{5i}(w)l^{\pm}_{3i^{\prime}}(z)=b_{35}(\frac{z}{w})l^{\pm}_{55}(z)l^{\pm}_{34}(w)+a_{35}(\frac{z}{w})l^{\pm}_{35}(z)l^{\pm}_{54}(w),
\end{align*}
we arrive at the relation
\begin{align}
\sum\limits_{i=5}^{8}c_{i4}(\frac{z}{w})e^{\pm}_{5i}(w)k^{\pm}_{3}(z)e^{\pm}_{3i^{\prime}}(z)=b_{35}(\frac{z}{w})k^{\pm}_{5}(w)k^{\pm}_{3}(z)e^{\pm}_{34}(w).
\end{align}
Using (6.58) and (6.59), due to the reltaion $e^{\pm}_{45}(z)=f^{\pm}_{54}(z)=0$, we obtain
\begin{align*}
\begin{split}
k^{\pm}_{5}(w)k^{\pm}_{3}(z)e^{\pm}_{34}(z)\Bigl(c_{55}(\frac{z}{w})-c_{54}(\frac{z}{w})\Bigr)&
=a_{35}(\frac{z}{w})k^{\pm}_{3}(z)e^{\pm}_{34}(z)k^{\pm}_{5}(w)\\
&\quad\,-b_{35}(\frac{z}{w})k^{\pm}_{5}(w)k^{\pm}_{3}(z)e^{\pm}_{34}(w),
\end{split}
\end{align*}
finally, we get
\begin{align*}
k^{+}_5(w)X^{+}_3(z)=\frac{rs(z_+-w)}{z_+s-rw}X^{+}_3(z)k^{+}_5(w).
\end{align*}

This completes the proof.
\end{proof}
\begin{lemma}
\begin{align}
k^{\pm}_4(z)k^{\pm}_5(w)&=k^{\pm}_5(w)k^{\pm}_4(z),\\
k^{\pm}_4(z)k^{\mp}_5(w)\frac{z_{\pm}-w_{\mp}}{rz_{\pm}-sw_{\mp}}&=\frac{z_{\mp}-w_{\pm}}{rz_{\mp}-sw_{\pm}}k^{\mp}_5(w)k^{\pm}_4(z).
\end{align}
\end{lemma}
\begin{proof}
The proof is similar to that of Lemma 6.5.
\end{proof}
\begin{lemma}
The following equations hold in $\mathcal{U}(\hat R)$:
\begin{align}
k^{\pm}_5(z)k^{\pm}_5(w)&=k^{\pm}_5(w)k^{\pm}_5(z),\\
k^{\pm}_5(z)k^{\mp}_{5}(w)&=k^{\mp}_{5}(w)k^{\pm}_5(z),\\
k^{\pm}_5(w)X^{+}_4(z)&=\frac{rs(w-z_{\pm})}{sw-rz_{\pm}}X^{+}_4(z)k^{\pm}_5(w),\\
X^{-}_4(z)k^{\pm}_5(w)&=\frac{rs(w-z_{\mp})}{sw-rz_{\mp}}k^{\pm}_5(w)X^{-}_4(z).
\end{align}
\end{lemma}
\begin{proof}
Here we only prove (6.64) as the other relations can be shown similarly. By (5.10) and (5.12), we have $M_{35}=M^{\prime}_{35}$, then we can get the relation:
\begin{align*}
l^{\pm}_{55}(w)l^{\pm}_{35}(z)=b_{35}(\frac{z}{w})l^{\pm}_{55}(z)l^{\pm}_{35}(w)+a_{35}(\frac{z}{w})l^{\pm}_{35}(z)l^{\pm}_{55}(w).
\end{align*}
By straightforward calculations, one checks that
\begin{align*}
k^{+}_5(w)X^{+}_4(z)&=\frac{rs(z_{+}-w)}{rz_{+}-sw}X^{+}_4(z)k^{+}_5(w).
\end{align*}

This completes the proof.
\end{proof}
\end{theorem}
\begin{proposition}
The following equations hold in $\mathcal{U}(\hat R)$
\begin{align}
&\Bigl\{X^{-}_{2}(z_1)X^{-}_{2}(z_2)X^{-}_{4}(w)-(r+s)X^{-}_{2}(z_1)X^{-}_{4}(w)X^{-}_{2}(z_2)\nonumber\\
&\quad+rsX^{-}_{4}(w)X^{-}_{2}(z_1)X^{-}_{2}(z_2)\Bigr\}+\Bigl\{z_1 \leftrightarrow z_2\Bigr\}=0,\\
&\Bigl\{X^{+}_{4}(z_1)X^{+}_{4}(z_2)X^{+}_{2}(w)-(r+s)X^{+}_{4}(z_1)X^{+}_{2}(w)X^{+}_{4}(z_2)\nonumber\\
&\quad+rsX^{+}_{2}(w)X^{+}_{4}(z_1)X^{+}_{4}(z_2)\Bigr\}+\Bigl\{z_1 \leftrightarrow z_2\Bigr\}=0,\\
&\Bigl\{rsX^{+}_{2}(z_1)X^{+}_{2}(z_2)X^{+}_{4}(w)-(r+s)X^{+}_{2}(z_1)X^{+}_{4}(w)X^{+}_{2}(z_2)\nonumber\\
&\quad+X^{+}_{4}(w)X^{+}_{2}(z_1)X^{+}_{2}(z_2)\Bigr\}+\Bigl\{z_1 \leftrightarrow z_2\Bigr\}=0,\\
&\Bigl\{rsX^{-}_{4}(z_1)X^{-}_{4}(z_2)X^{-}_{2}(w)-(r+s)X^{-}_{4}(z_1)X^{-}_{2}(w)X^{-}_{4}(z_2)\nonumber\\
&\quad+X^{-}_{2}(w)X^{-}_{4}(z_1)X^{-}_{4}(z_2)\Bigr\}+\Bigl\{z_1 \leftrightarrow z_2\Bigr\}=0.
\end{align}
\end{proposition}
\begin{proof}
The equations can be proved similarly as in \cite{JL}.
\end{proof}
Now we proceed to the case of general $n$. We first restrict the relation to $E_{ij}\otimes E_{kl}$, $2\le i,j,k,l\le 2n-1$. By induction, we get all the commutation relations we need except those between $X^{\pm}_{1}(z)$, $k^{\pm}_1(z)$, and $X^{\pm}_{n}(z)$, $k^{\pm}_{n+1}(z)$.
\begin{lemma}
The following equations hold in $\mathcal{U}(\hat R)$:
\begin{align}
k^{\pm}_1(z)X^{+}_{n}(w)&=rsX^{+}_{n}(w)k^{\pm}_1(z),\\
rsk^{\pm}_1(z)X^{-}_{n}(w)&=X^{-}_{n}(w)k^{\pm}_1(z),\\
X^{\pm}_n(w)X^{\pm}_1(z)&=X^{\pm}_1(z)X^{\pm}_n(w),\\
X^{\pm}_n(w)X^{\mp}_1(z)&=X^{\mp}_1(z)X^{\pm}_n(w),\\
k^{\pm}_1(z)k^{\pm}_{n+1}(w)&=k^{\pm}_{n+1}(w)k^{\pm}_1(z),
\\
k^{\pm}_{n+1}(z)k^{\pm}_{n+1}(w)&=k^{\pm}_{n+1}(w)k^{\pm}_{n+1}(z),\\
k^{\pm}_{n+1}(z)k^{\mp}_{n+1}(w)&=k^{\mp}_{n+1}(w)k^{\pm}_{n+1}(z),\\
k^{\pm}_{n+1}(z)X^{+}_n(w)&=\frac{rs(w-z_{\pm})}{sw-rz_{\pm}}X^{+}_n(w)k^{\pm}_{n+1}(z),\\
X^{-}_n(w)k^{\pm}_{n+1}(z)&=\frac{rs(w-z_{\mp})}{sw-rz_{\mp}}k^{\pm}_{n+1}(z)X^{-}_n(w),\\
\frac{w_{\pm}-z_{\mp}}{w_{\pm}r-z_{\mp}s}k^{\mp}_{n+1}(w)k^{\pm}_1(z)&=\frac{w_{\mp}-z_{\pm}}{w_{\mp}r-z_{\pm}s}k^{\pm}_1(z)k^{\mp}_{n+1}(w),\\
[X^{+}_n(z),X^{-}_n(w)]&=(s^{-1}-r^{-1})\Bigl\{\delta\Bigl(\frac{z_-}{w_+}\Bigr)k^-_{n+1}(w_+)k^-_{n-1}(w_+)^{-1}\nonumber\\
&\qquad-\delta\Bigl(\frac{z_+}{w_-}\Bigr)k^+_{n+1}(z_+)k^+_{n-1}(z_+)^{-1}\Bigr\}.
\end{align}
\end{lemma}
\begin{proof}
By straightforward calculations one checks that the preceding formulas are correct.
\end{proof}
Finally, we define the map $\tau: U_{r,s}\mathcal(\widehat {\mathfrak{so}}_{2n})\rightarrow \mathcal{U}(\hat R)$ as follows:
\begin{align*}
x^{\pm}_i(z)&\mapsto (r-s)^{-1}X^{\pm}_i(z(rs^{-1})^{\frac{i}{2}}),
\\
x^{\pm}_n(z)&\mapsto (r-s)^{-1}X^{\pm}_n(z(rs^{-1})^{\frac{n-1}{2}}),
\\
\varphi_i(z)&\mapsto k^{+}_{i+1}(z(rs^{-1})^{\frac{i}{2}})k^+_i(z(rs^{-1})^{\frac{i}{2}})^{-1},
\\
\psi_i(z)&\mapsto k^{-}_{i+1}(z(rs^{-1})^{\frac{i}{2}})k^-_i(z(rs^{-1})^{\frac{i}{2}})^{-1},\\
\varphi_n(z)&\mapsto k^{+}_{n+1}(z(rs^{-1})^{\frac{n-1}{2}})k^+_{n-1}(z(rs^{-1})^{\frac{n-1}{2}})^{-1},\\
\psi_n(z)&\mapsto k^{-}_{n+1}(z(rs^{-1})^{\frac{n-1}{2}})k^-_{n-1}(z(rs^{-1})^{\frac{n-1}{2}})^{-1},
\end{align*}
where $1\le i \le n-1$, and satisfy all the relations of Proposition 6.13.
\begin{proposition}
In $U_{r,s}\mathcal(\widehat {\mathfrak{so}_{2n}})$, the generating series $x^{\pm}_{i}(z), \varphi_i(z)$, $\psi_i(z)$, and $x^{\pm}_{n}(z), \varphi_n(z)$, $\psi_n(z)$, and the relations between $x^{\pm}_{i}(z), \varphi_i(z)$, and $\psi_i(z)$ are the same as in $U_{r,s}\mathcal(\widehat {\mathfrak{sl}_{n}})$, the other relations as follows.
\begin{align}
[\varphi_j(z),\varphi_n(w)]&=0,\quad [\psi_j(z),\psi_n(w)]=0,\\
\varphi_j(z)\psi_n(w)&=\frac{g_{jn}\Bigl(\frac{z_-}{w_+}\Bigr)}{g_{jn}\Bigl(\frac{z_+}{w_-}\Bigr)}\psi_n(w)\varphi_j(z),\quad 1\le j\le n,\\
\varphi_n(z)x^{\pm}_{n-2}(w)&=(rs)^{\pm1}g_{n,n-2}\Bigl(\frac{z}{w_{\pm}}\Bigr)^{\pm1}x^{\pm}_{n-2}(w)\varphi_n(z),\nonumber\\
\psi_n(z)x^{\pm}_{n-2}(w)&=(rs)^{\mp1}g_{n,n-2}\Bigl(\frac{w_{\mp}}{z}\Bigr)^{\mp1}x^{\pm}_{n-2}(w)\varphi_n(z),\\
\varphi_n(z)x^{\pm}_{n}(w)&=g_{nn}\Bigl(\frac{z}{w_{\pm}}\Bigr)^{\pm1}x^{\pm}_{n}(w)\varphi_n(z),\nonumber\\
\psi_n(z)x^{\pm}_{n}(w)&=g_{nn}\Bigl(\frac{w_{\mp}}{z}\Bigr)^{\mp1}x^{\pm}_{n}(w)\varphi_n(z),\\
\varphi_n(z)x^{\pm}_{n-1}(w)&=(rs)^{\pm1}x^{\pm}_{n-1}(w)\varphi_n(z),\nonumber\\
\psi_n(z)x^{\pm}_{n-1}(w)&=(rs)^{\mp1}x^{\pm}_{n-1}(w)\psi_n(z),\\
\varphi_n(z)x^{\pm}_{l}(w)&=x^{\pm}_{l}(w)\varphi_n(z),\nonumber
\\
\psi_n(z)x^{\pm}_{l}(w)&=x^{\pm}_{l}(w)\psi_n(z),\quad 1\le l\le n-3,\\
\varphi_t(z)x^{\pm}_{n}(w)&=x^{\pm}_{n}(w)\varphi_t(z),\nonumber\\
\psi_t(z)x^{\pm}_{n}(w)&=x^{\pm}_{n}(w)\varphi_t(z),\quad 1\le t\le n-3,\nonumber\\
\varphi_{n-2}(z)x^{\pm}_{n}(w)&=g_{n,n-2}\Bigl(\frac{z}{w_{\pm}}\Bigr)^{\pm1}x^{\pm}_{n}(w)\varphi_{n-2}(z),\nonumber\\
\psi_{n-2}(z)x^{\pm}_{n}(w)&=g_{n,n-2}\Bigl(\frac{w_{\mp}}{z}\Bigr)^{\mp1}x^{\pm}_{n}(w)\varphi_{n-2}(z),\\
x^{\pm}_{n-2}(z)x^{\pm}_{n}(w)&=g_{n,n-2}\Bigl(\frac{z}{w}\Bigr)^{\pm1}x^{\pm}_{n}(w)x^{\pm}_{n-2}(z),\nonumber\\
x^{\pm}_{i}(z)x^{\pm}_{n}(w)&=\langle w^{\prime}_n,w_i\rangle^{\pm1}x^{\pm}_{n}(w)x^{\pm}_{i}(z),\quad a_{in}=0
\\
[x^{+}_{n}(z),x^{-}_{j}(w)]&=(s^{-1}{-}r^{-1})\delta_{jn}\Bigl\{\delta\Bigl(\frac{z_{-}}{w_+}\Bigr)\psi_n(w_+){-} \delta\Bigl(\frac{z_{+}}{w_-}\Bigr)\varphi_n(z_+)\Bigr\},\ j\le n,
\\
\Bigl\{x^{\pm}_{n-2}(z_1)x^{\pm}_{n-2}&(z_2)x^{\pm}_{n}(w)-(r^{\pm1}+s^{\pm1})x^{\pm}_{n-2}(z_1)x^{\pm}_{n}(w)x^{\pm}_{n-2}(z_2)\nonumber\\
&\qquad\quad+(rs)^{\pm1}x^{\pm}_{n}(w)x^{\pm}_{n-2}(z_1)x^{\pm}_{n-2}(z_2)\Bigr\}+\Bigl\{z_1 \leftrightarrow z_2\Bigr\}=0,
\end{align}
\begin{align}
\Bigl\{x^{\pm}_{n}(z_1)x^{\pm}_{n}&(z_2)x^{\pm}_{n-2}(w)-(r^{\mp1}+s^{\mp1})x^{\pm}_{n}(z_1)x^{\pm}_{n-2}(w)x^{\pm}_{n}(z_2)\nonumber\\
&\qquad\quad+(rs)^{\mp1}x^{\pm}_{n-2}(w)x^{\pm}_{n}(z_1)x^{\pm}_{n}(z_2)\Bigr\}+\Bigl\{z_1 \leftrightarrow z_2\Bigr\}=0,
\end{align}
where $z_+=zr^{\frac{c}{2}}$ and $z_-=zs^{\frac{c}{2}}$, we set $g^{\pm}_{ij}(z)=\sum\limits_{n\in\mathbb{Z}_{+}}c^{\pm}_{ijn}z^{n}$,
a formal power series in $z$, the expression is as follows:
$$g^{\pm}_{ij}(z)=\frac{\langle w^{\prime}_j,w_i\rangle^{\pm1}z-(\langle w^{\prime}_j,w_i\rangle \langle w^{\prime}_i,w_j\rangle^{-1})^{\pm\frac{1}{2}}}{z-(\langle w^{\prime}_i,w_j\rangle \langle w^{\prime}_j,w_i\rangle)^{\pm\frac{1}{2}}}.$$
\end{proposition}

\bigskip
\centerline{\bf ACKNOWLEDGMENTS}

\bigskip

We thank Professors Zhengwei Liu, Jinsong Wu, and Yilong Wang for inviting Hu to report the current results in the international workshop on Advances in Quantum Algebras during Jan. 23--27, 2024 in BIMSA (Beijing Institute of Mathematical Sciences and Applications). We also thank ECNU for partial support of Xu's master thesis where the basic $R$-matrices were announced in the final version on March 8, 2023.

\bibliography{mybibfile}

\begin{thebibliography}{99}

\bibitem[BW]{BW}
G. Benkart, S. Witherspoon, Two-parameter quantum groups and Drinfel'd doubles, Algebr. Represent. Theory 7 (2004), 261---286.

\bibitem[BGH]{BGH}
N. Bergeron, Y. Gao, N. Hu, Drinfel'd doubles and Lusztig's symmetries of two-parameter quantum groups, J. Algebra  301 (2006), 378---405.

\bibitem[BGH1]{BGH1}
---, Representations of two-parameter quantum orthogonal groups and symplectic groups, In: Proceedings of the International Conference on Complex Geometry and Related Fields, AMS/IP, Stud. Adv. Math., Vol. 39, pp. 1---21, Amer. Math. Soc., Providence, RI, 2007.


\bibitem[CHW]{Chen-Hu-Wang} X. Y. Chen, N. Hu, X. L. Wang, Convex PBW-type Lyndon bases and restricted two-parameter quantum group of type $F_4$, Acta Math. Sinica \textbf{39} (6), (2023), 1053---1084.

\bibitem[DF]{DF}
J. Ding, I.B. Frenkel, Isomorphism of two realizations of quantum affine algebra $U_q(\widehat{\mathfrak{gl}}_n)$, Comm. Math. Phys. 156 (1993), 277---300.

\bibitem[D1]{D1}
V. Drinfeld, Hopf algebras and the quantum Yang-Baxter equation, Sov. Math., Dokl. 32 (1985), 254---258.

\bibitem[D2]{D2}
---, Quantum group, in: Proc. ICM, Vol. 1, 2, Berkeley, Calif., 1986, Amer. Math. Soc. Providence, RI, 1987, pp. 798---820.

\bibitem[D3]{D3}
---, A new realization of Yangians and quantized affine algebras, Sov. Math., Dokl. 36 (1988), 212---216.

\bibitem[FRT90]{FRT90}
L. Faddeev, N. Reshetikhin, L.  Takhtajan, Quantization of Lie groups and Lie algebras, Leningrad Math. J. 1 (1990), 193---225.

\bibitem[FRT89]{FRT89}
---, Quantization of Lie groups and Lie algebras, in: Yang-Baxter Equations and Integrable Systems, in: Advanced Series in Mathematical Physics, vol. 10, World Scientific, Singapore, 1989, pp. 299---309.


\bibitem[Ga]{Ga}
H. Garland, The arithmetic theory of loop groups, Publ. Math. IHES 52 (1980), 5---136.

\bibitem[HP]{HP}
N. Hu, Y. Pei, Notes on two-parameter quantum groups, (II), Comm. Algebra 40 (9) (2012), 3202---3220.

\bibitem[HRZ]{HRZ}
N. Hu, M. Rosso, H.L. Zhang, Two-parameter quantum affine algebra $U_{r,s}(\widehat{\mathfrak{sl}}_n)$, Drinfeld realization and quantum affine Lyndon basis, Comm. Math. Phys. 278 (2008), 453---486.

\bibitem[HXZ]{HXZ} N. Hu, X. Xu, R.S. Zhuang, $RLL$-realization of two-parameter quantum affine algebra in type $B^{(1)}_n$, arXiv:2405.06587, 29 pages.

\bibitem[XZ]{HZ}
N. Hu, H.L. Zhang, Generating functions with $\tau$-invariance and vertex representations of quantum affine algebras $U_{r,s}(\hat{\mathfrak{g}})$ (I): simply-laced cases, 	arXiv: 1401.4925.




\bibitem[J]{J}
M. Jimbo, A $q$-difference analogue of $U(\mathfrak{g})$ and the Yang-Baxter equation, Lett. Math. Phys. 10 (1985), 63---69.

\bibitem[JL]{JL}
N. Jing, M. Liu, $R$-matrix realization of two-parameter quantum affine algebra $U_{r,s}(\widehat{\mathfrak{gl}}_n)$, J. Algebra 488 (2017), 1---28.

\bibitem[JLM]{JLM}
N. Jing, M. Liu, A. Molev, Isomorphism between the $R$-matrix and Drinfeld presentations of quantum affine algebra: types $B$ and $D$, SIGMA Symmetry Integrability Geom. Methods Appl.16 (2020), Paper No. 043, 49 pp.

\bibitem[JLM1]{JLM1}
---, Isomorphism between the $R$-matrix and Drinfeld presentations of quantum affine algebra: type $C$, J. Math. Phys. 61 (3) (2020),  031701, 41 pp.

\bibitem[KS]{KS}
A. Klimyk, K. Schm\"udgen, Quantum groups and their representations, Texts and Monographs in Physics, Springer-Verlag, Berlin, 1997, xx+552 pp.

\bibitem[RS]{RS}
N.Yu. Reshetikhin, M.A. Semenov-Tian-Shanski, Central extensions of quantum current groups, Lett. Math. Phys. 19 (1990), 133---142.

\bibitem[T]{T}
M. Takeuchi, A two-parameter quantization of $GL(n)$, Proc. Japan Acad. Ser. A 66 (1990), 112---114.

\bibitem[ZhHJ]{ZhHJ} X. Zhong, N. Hu, N. Jing, $RLL$-realization of two-parameter quantum affine algebra in type $C^{(1)}_n$, arXiv:2405.06597, 25 pages.


\end{thebibliography}

\end{document}